\documentclass[10pt, one column, draft cls]{ieeeconf}
\IEEEoverridecommandlockouts 

\usepackage{amsmath,amssymb,bbm,color,psfrag,amsfonts,mathtools}
\usepackage{dsfont}
\usepackage{graphicx,cuted}                       
\usepackage[pdftex,pdfpagelabels,bookmarks,hyperindex,hyperfigures]{hyperref}
\usepackage{wrapfig}
\usepackage{pdfpages}
\usepackage[ruled,lined]{algorithm2e} 
\usepackage{algpseudocode}
   \SetKwInOut{Input}{input}\SetKwInOut{Output}{output}
\usepackage{accents}
\usepackage{lipsum}

\usepackage{paralist}   

\usepackage{booktabs}  
\usepackage{multirow}
\usepackage[normalem]{ulem}

\usepackage{caption}
\usepackage{subcaption}

\usepackage{cite}

\newcounter{problem}

\newtheorem{theorem}{Theorem}
\newtheorem{corollary}{Corollary}
\newtheorem{lemma}{Lemma}
\newtheorem{proposition}{Proposition}

\newtheorem{definition}{Definition}
\newtheorem{assumption}{Assumption}

\newtheorem{remark}{Remark}

\usepackage{color}

\newcommand{\real}{\mathbb{R}}

\newcommand{\argmin}{\mathop{\rm argmin}}

\newcommand{\setdef}[2]{\left\{#1 \; | \; #2\right\}}

\newcommand{\NN}{\mathbb{N}}
\newcommand{\mc}{\mathcal}

\newcommand{\onebf}{\mathbf{1}}
\newcommand{\until}[1]{\{1,\dots,#1\}}

\newcommand{\map}[3]{#1: #2 \rightarrow #3}

\newcommand{\myk}{q}
\newcommand{\mykk}{r}

\algnewcommand{\algorithmicgoto}{\textbf{go to}}%
\algnewcommand{\Goto}[1]{\algorithmicgoto~\ref{#1}}%

\usepackage{etoolbox}
\newtoggle{long}
\toggletrue{long}

\usepackage{tikz}
\usetikzlibrary{positioning}
\usetikzlibrary{arrows}
\usepackage{verbatim}

\title{Queue Length Simulation for Signalized Arterial Networks \\ and Steady State Computation under Fixed Time Control}
\author{Pouyan Hosseini\thanks{The authors are with the Sonny Astani Department of Civil and Environmental Engineering at the University of Southern California, Los Angeles, CA. \texttt{\{seyedpoh,ksavla\}@usc.edu}. This work was supported in part by USDOT \# DTRT 13-G-UTC57 and NSF CAREER ECCS \# 1454729.} \qquad Ketan Savla}

\date{\today}

\begin{document}
\maketitle

\begin{abstract}
We consider traffic flow dynamics for a network of signalized intersections, where the 
outflow from every link is constrained to be equal to a given capacity function if the queue length is positive, and equal to the minimum of cumulative inflow and capacity function otherwise. In spite of the resulting dynamics being discontinuous, recent work has proved existence and uniqueness of the resulting queue length trajectory if the inter-link travel times are strictly bounded away from zero. The proof, which also suggests a constructive procedure, relies on showing desired properties on contiguous time intervals of length equal to the minimum among all link travel times. We provide an alternate framework to obtain queue length trajectories by direct simulation of delay differential equations, where link outflows are obtained from the provably unique solution to a linear program. Existence and uniqueness of the solution to the proposed model for traffic flow dynamics is established for piecewise constant external inflow and capacity functions, and the proposed method does not require travel times to be bounded away from zero. Additionally, if the external inflow and capacity functions are periodic and satisfy a stability condition, then there exists a globally attractive periodic orbit. We provide an iterative procedure to compute this periodic orbit. A periodic trajectory is iteratively updated for every link based on updates to a specific time instant when its queue length transitions from being zero to being positive. The update for a given link is based on the periodic trajectories computed in the previous iteration for its upstream links. The resulting iterates are shown to converge uniformly monotonically to the desired periodic orbit. 
\end{abstract}

\section{Introduction}
Modeling of traffic flow dynamics for signalized arterial networks has to strike a tradeoff between the ability to capture variations induced by alternating red/green phases and computational complexity of the resulting framework for the purpose of performance evaluation and control synthesis. 
Store-and-forward models, e.g., see \cite{Papageorgiou:03}, approximate the dynamics by replacing a time-varying outflow due to alternating green and red phase on a link with an \emph{equivalent} average outflow. Such models have been used for optimal green time split control, e.g., see \cite{Dans.Gazis:76,Aboudolas.et.al:09}. Continuous-time versions of these models have also been used for green time control, e.g., in \cite{Savla.Lovisari.ea.Allerton13}. 
However, the approximation does not model the effect of offsets and cycle lengths. These limitations are overcome by discrete-event models, which have been utilized for optimal control synthesis for isolated signalized intersections in some cases, e.g., see \cite{DenBoom.DeSchutter:06,Haddad.DeSchutter.ea:TAC10}.  


\cite{Muralidharan.Pedarsani.ea:15} proposed and analyzed a model, which captures offset and cycle times in the same spirit as discrete-event models. In particular, in \cite{Muralidharan.Pedarsani.ea:15}, a fixed-time control setting is considered, where every link is endowed with a given capacity function, that specifies the maximum possible outflow from a link as a function of time. In order to maintain non-negativity of queue lengths, the outflow from every link is constrained to be equal to the capacity function if the queue length is positive, and equal to the minimum of cumulative inflow and capacity function otherwise. In spite of the resulting dynamics being discontinuous, it was shown in \cite{Muralidharan.Pedarsani.ea:15} that the traffic dynamics admits a unique queue length trajectory if the inter-link travel times are strictly bounded away from zero. The proof, which also suggests a constructive procedure, relies on showing desired properties on contiguous time intervals of length equal to the minimum among all inter-link travel times. 

We provide an alternate framework to obtain queue length trajectories by direct simulation of delay differential equations, where link outflows are obtained from the provably unique solution to a linear program. For given queue lengths, this linear program solves for maximum cumulative outflow from all links subject to constraints imposed by the link capacity functions, and subject to maintaining non-negativity of queue lengths. 
Existence and uniqueness of the solution to delay differential equations is established for piecewise constant external inflow and capacity functions, and the method does not require travel times to be bounded away from zero. The existence and uniqueness result also extends to adaptive control policies, as long as the resulting capacity functions remain piecewise constant. This would happen, e.g., if traffic signal control parameters (green time, cycle length, and offsets) at every intersection are updated once per cycle. The piecewise constant assumption is practically justified because a common model for a capacity function is that it is equal to the saturated capacity during the green phase and zero otherwise, and external inflows can be modeled as a sequence of rectangular pulses representing arriving vehicle platoons. 
The key idea in the proof is that, under constant inflow and capacity, the set of links with zero queue lengths is monotonically non-decreasing, which implies overall finite discontinuities over any given time interval under the piecewise constant assumption. The ability to model zero inter-link travel time is particularly desirable for possible extensions to model finite queue capacity, under which inter-link travel time approaches zero as the downstream queue approaches capacity. 

If, additionally, the external inflow and capacity functions are periodic and satisfy a stability condition, then there exists a globally attractive periodic orbit. This result and its proof follows the same structure as in \cite{Muralidharan.Pedarsani.ea:15}, but is adapted to the proposed modeling framework. 
One consequence of this adaptation is that we work with the $\ell_1$ norm, instead of the sup norm in \cite{Muralidharan.Pedarsani.ea:15}, for continuity arguments in our proofs. 

Our most novel contribution is a procedure to explicitly calculate the globally attractive periodic orbit. Indeed, this was noted as an important ``outstanding open problem" in \cite{Muralidharan.Pedarsani.ea:15}, due to its usefulness in directly quantifying relevant performance metrics for a given fixed-time control. 
We provide an iterative procedure to compute this periodic orbit. A periodic trajectory is iteratively updated for every link based on updates to a specific time instant when its queue length transitions from being zero to being positive. This update for a given link is based on the periodic trajectories computed in the previous iteration for upstream links. The resulting iterates are shown to converge uniformly monotonically to the desired periodic orbit.

The representation of periodic orbit in terms of the time instants when queue length transitions between being positive and zero, as is implicit in our computational procedure, is to be contrasted with sinusoidal approximation consisting of a single harmonic proposed, e.g., in \cite{Coogan.Kim.ea:17}. While it is compelling to improve this approximation by including higher harmonics~\cite{Gartner.Deshpande:09}, such an approach can potentially face several challenges: computing Fourier coefficients is not easy due to discontinuous dynamics; no bounds exist on approximation error for a given number of harmonics; and most importantly, because of discontinuity, including arbitrarily high number of harmonics may not give a zero approximation error due to the well-known Gibbs phenomenon. On the other hand, our proposed procedure computes the periodic orbit with arbitrary accuracy.   
 
In summary, the key contributions of the paper are as follows. First, we provide a delay differential equation framework to directly simulate queue length dynamics under 
fixed-time or adaptive control, by establishing that it has a unique solution as long as the external inflow and capacity functions are piecewise constant. Second, under additional periodicity and stability condition, we adapt a recently proposed technique to establish existence of a globally attractive periodic orbit in our setting. Third, we provide a procedure to compute this periodic orbit with arbitrary accuracy. Illustrative simulations, including comparison with steady-state queue lengths from a microscopic traffic simulator, are also included.

The outline of the paper is as follows. Section~\ref{sec:problem-formulation} contains the proposed delay differential equation framework to simulate queue length dynamics. Section~\ref{sec:steady-state-single-link} provides the (non-iterative) framework to compute the periodic orbit for an isolated link. This forms the basis for an iterative procedure to compute periodic orbits for a network in Section~\ref{sec:steady-state-network} where we also establish uniform monotonic convergence of the iterates to the desired periodic orbit. Section~\ref{sec:simulations} presents illustrative simulation results and concluding remarks are presented in Section~\ref{sec:conclusions}. The proofs for most of the technical results are collected in the Appendix. 

We conclude this section by introducing key concepts and notations to be used throughout the paper. $\real$, $\real_{\geq 0}$, $\real_{>0}$, $\real_{\leq 0}$ and $\real_{< 0}$ will stand for real, non-negative real, strictly positive real, non-positive real, and strictly negative real, respectively, and $\NN$ denotes the set of natural numbers. 
For $x \in \real$, we let $[x]^+=\max \{x,0\}$ denote the non-negative part of $x$. 
A function $\map{f}{X \subsetneq \real}{\real^n}$ is called \emph{piece-wise constant} if it has only finitely many pieces, i.e., $X$ can be partitioned into a finite number of contiguous right-open sets over each of which $f$ is constant.
The road network topology is described by a directed multi-graph $\mc G=\left(\mc V,\mc E\right)$ with no self-loops, where $\mc V$ is the set of intersections and $\mc E$ is the set of directed links. 

\section{Problem Formulation}
\label{sec:problem-formulation}
\subsection{Traffic Flow Dynamics}
The network state at time $t$ is described by the vector of queue lengths, $x(t) \in \real_+^{\mc E}$ corresponding to the number of stationary vehicles, and the history of relevant past departures from the links, $\beta(t)$, which quantifies the number of vehicles traveling in between links. The quantity $\beta(t)$ shall be described formally soon.
Let $\map{c_i}{\real_{\geq 0}}{\real_{\geq 0}}$ and $\map{\lambda_i}{\real_{\geq 0}}{\real_{\geq 0}}$ be saturated flow capacity and external inflow functions, respectively, for link $i \in \mc E$. 
Let the matrix $R \in \real_{\geq 0}^{\mc E \times \mc E}$ denote the routing of flow, e.g., $R_{ji}$ denotes the fraction of flow departing link $j$ that gets routed to link $i$. Naturally $R_{ji}=0$ if link $i$ is not immediately downstream to link $j$. We shall assume that $R$ is sub-stochastic, i.e., all of its entries are non-negative, all the row sums are upper bounded by 1, and there is at least one row 
whose row sum is strictly less than one. We further assume the following on the connectivity of $\mc G$.

\begin{assumption}
\label{ass:connectivity}
\begin{enumerate}
\item[(i)] $\mc G$ is weakly connected, i.e., for every $i, j \in \mc E$, there exists a directed path in $\mc E$ from $i$ to $j$, or from $j$ to $i$.
\item[(ii)] For every $i \in \mc E$, either the sum of entries of the $i$-th row in $R$ is strictly less than one, or there exists a directed path from $i$ to at least one link $j$ such that the entries of the $j$-th row in $R$ is strictly less than one.
\end{enumerate}
\end{assumption}

\begin{remark}
\label{ass:routing-matrix}
The weak connectivity aspect of 
Assumption~\ref{ass:connectivity} is without loss of generality: if $\mc G$ is not weakly connected, then our analysis applies to each connected component of $\mc G$, as long as each of these connected components satisfies (ii) in Assumption~\ref{ass:connectivity}. Indeed, part (ii) of Assumption~\ref{ass:connectivity} implies that, for every vehicle arriving into the network, either it is possible for the vehicle to depart directly from the arrival link, or there exists a directed path to an another link from which the vehicle can depart the network. 
Formally, part (ii) of Assumption~\ref{ass:connectivity} implies that 
the spectral radius of $R$, and hence also of $R^T$, is strictly less than one. In particular, this guarantees that $I-R^T$ is invertible. 
%
\end{remark}

%

We now describe a model for traffic flow dynamics. The queue length dynamics is described by a standard mass balance equation: for $t \geq 0$,
$$
\dot{x}_i(t) = \lambda_i(t) + \sum_{j \in \mc E} R_{ji} z_j\left(t-\delta_{ji}\right) - z_i(x(t),t), \qquad i \in \mc E
$$
where $z_i(x(t),t)$ denotes the outflow from link $i$ at time $t$. In \eqref{eq:flow-dynamics}, $\delta_{ji} \geq 0$ is the travel time from link $j$ to $i$, and $z_i(t-\delta_{ji})$ is  a concise notation for $z_i(x(t-\delta_{ji}),t-\delta_{ji})$. It would be convenient to rewrite the queue length dynamics as: for $t \geq 0$,
\begin{subequations}
\label{eq:flow-model-main}
\begin{equation}
\label{eq:flow-dynamics}
\dot{x}_i(t) = \tilde{\lambda}_i(t) + \sum_{j \in \mc E_i} R_{ji} z_j(x(t),t) - z_i(x(t),t), \qquad i \in \mc E
\end{equation}
\end{subequations}
\setcounter{equation}{1}
where 
\begin{equation}
 \label{eq:lambda-tilde-def}
 \tilde{\lambda}_i(t):=\lambda_i(t)+\sum_{j \in \mc E \setminus \mc E_i} R_{ji} z_j(t-\delta_{ji}), \quad i \in \mc E
\end{equation} 
is the net inflow to link $i$ due to external arrivals and arrivals due to vehicles from upstream which were traveling until $t$, and 
\begin{equation}
\label{eq:E-i-def-new}
\mc E_i:=\setdef{j \in \mc E}{R_{ji} > 0 \quad \& \quad \delta_{ji} = 0}
\end{equation}
is the set of links upstream of $i$ with zero inter-link travel time. Let $\bar{\delta}_j:=\max\{\delta_{ji}: i \in \mc E, \, R_{ji} > 0\}$ be the maximum among all travel times from link $j$ to its downstream links. 
We let 
\begin{equation}
\tag{1b}
\label{eq:beta-def}
\beta(t):=\{z_j(s): s \in [t-\bar{\delta}_j, t)\}_{j \in \mc E}
\end{equation}
be the history of relevant past departures\footnote{If $\bar{\delta}_j=0$ for some link $j$, then the departure history from such a link is not included in $\beta(t)$.}, and $\|\beta(t)\|_1:=\sum_{i \in \mc E} \sum_{j \in \mc E} R_{ji} \int_{t-\delta_{ji}}^t z_j(s) \, ds$ be the number of vehicles traveling in between links at time $t$.\footnote{It is easy to verify that this definition of $\|\beta(t)\|_1$ satisfies all the properties of a norm.} 
Finally, let $x(t):=\{x_i(t)\}_{i \in \mc E}$, $z(x(t),t) \equiv z(t) := \{z_i(x(t),t)\}_{i \in \mc E}$, $\lambda(t) := \{\lambda_i(t)\}_{i \in \mc E}$, and $c(t):=\{c_i(t)\}_{i \in \mc E}$ denote the collection of corresponding quantities over all links. \eqref{eq:flow-dynamics}-\eqref{eq:beta-def} collectively describe the evolution of $(x(t),\beta(t))$ starting from initial condition $(x(0),\beta(0))$. We propose link outflows $z(x(t),t)$ for $t \geq 0$ be obtained as solution to the following linear program, for any $\eta \in \real_{>0}^{\mc E}$:
\begin{equation}
\begin{aligned}
& \underset{z \in \real^{\mc E}}{\text{maximize}} \qquad \eta^{T} z \\
& \text{subject to} \qquad \qquad  z_i \leq c_i(t), \qquad i \in \mc E   \\
&  \qquad \qquad \qquad  z_i \leq \tilde{\lambda}_i(t) + \sum_{j \in \mc E_i} R_{ji} z_j, \qquad \text{if } i \in \mc I(x)  \\
\end{aligned}
\tag{1c}
\label{z-linprog}
\end{equation}
\setcounter{equation}{2}
where
$$
\mc I(x) :=\setdef{i \in \mc E}{x_i=0}
$$
is the set of links with no stationary vehicles. 
\eqref{z-linprog} computes the maximum cumulative outflow, weighted by $\eta$, in the network, subject to two constraints. The first one imposes link-wise capacity constraint, and the second one imposes the constraint that, for a link with zero queue length, its outflow is no greater than its inflow. The second constraint is to ensure non-negativity of queue lengths. The well-posedness of our proposed method for computing link outflows, i.e., uniqueness of solution to \eqref{z-linprog} for a given $\eta$, and independence w.r.t. $\eta$ is established in the next section. Thereafter, we establish existence and uniqueness of the solution to our traffic flow model in \eqref{eq:flow-dynamics}-\eqref{eq:beta-def}-\eqref{z-linprog}, which we shall collectively refer to as \eqref{eq:flow-model-main}.
 

%
%
In order to present our results on existence and uniqueness concisely, we introduce a couple of more notations. Let $\bar{\delta}:= \max_{(j,i) \in \mc E \times \mc E: \, R_{ji}>0} \delta_{ji}$ and $\underline{\delta}:= \min_{(j,i) \in \mc E \times \mc E: \, R_{ji}>0} \delta_{ji}$ be the, respectively, maximum and minimum among all inter-link travel times. 

\subsection{Existence of Solution to \eqref{eq:flow-dynamics}}
The proof of the next result is provided in the Appendix.


\begin{proposition}
\label{prop:solution-properties}
Given $(x(t),\beta(t))$, $\lambda(t)$, and $c(t)$, \eqref{z-linprog} has a unique solution, which is independent of $\eta \in \real_{>0}^{\mc E}$. Moreover, the optimal solution satisfies 
\begin{equation}
\label{eq:z-def}
z_i(x,t) = \begin{cases} 
c_i(t) &  i \in \mc E \setminus \mc I(x) \\
\min \left\{c_i(t), \tilde{\lambda}_i(t) + \sum_{j \in \mc E_i} R_{ji} z_j(x,t) \right\} & i \in \mc I(x)  
\end{cases}
\end{equation}
\end{proposition}
\vspace{0.1in}
Proposition~\ref{prop:solution-properties} implies that $z(x(t),t)$ in \eqref{eq:flow-dynamics} is well-defined. With regards to \eqref{eq:z-def}, indeed for $i \in \mc I(x)$, 
$z_i(x,t)=\tilde{\lambda}_i(t) + \sum_{j \in \mc E_i} R_{ji} z_j(x,t)$, except possibly at time instants when there is a change in $\mc I(x)$. 
It is rather straightforward to see that \eqref{eq:flow-model-main} admits a unique solution in between such changes. The frequency of such changes in general depends on $\lambda(t)$, $c(t)$, and the initial condition $\beta(0)$. We bound the frequency of changes, and thereby establish existence and uniqueness of the solution to \eqref{eq:flow-model-main} for all $t \geq 0$, under the following practical assumption.

\begin{assumption}
\label{ass:piece-wise-constant}
$\{\map{\lambda_i}{[0,T]}{\real_{\geq 0}^{\mc E}}\}_{i \in \mc E}$, $\{\map{c_i}{[0,T]}{\real_{\geq 0}^{\mc E}}\}_{i \in \mc E}$, and $\{\map{z_i}{[-\bar{\delta},0]}{\real_{\geq 0}^{\mc E}}\}_{i \in \mc E}$ are all piece-wise constant.
\end{assumption}

The proof of the next result is provided in Appendix.

\begin{proposition}
\label{prop:existence-traffic-dynamics-solution}
Let $\lambda(t)$, $c(t)$ and the initial condition $(x(0),\beta(0))$ satisfy Assumption~\ref{ass:piece-wise-constant}. Then, there exists a unique solution $(x(t),\beta(t)) \geq 0$ for all $t \geq 0$, to \eqref{eq:flow-model-main}.
\end{proposition}

\begin{remark}
\begin{enumerate}
\item \eqref{eq:flow-model-main} allows direct simulation of queue length dynamics under fixed-time control. Moreover, unlike \cite{Muralidharan.Pedarsani.ea:15}, existence of a unique solution to \eqref{eq:flow-model-main} does not require $\underline{\delta}$ to be strictly greater than zero. However, this comes at the expense of piecewise constant assumption.  
\item Assumption~\ref{ass:piece-wise-constant} is practically justified because a common model for a capacity function is such that it is equal to the saturated capacity during the green phase and zero otherwise, and external inflows, as well as past departures before $t=0$ can be modeled as a sequence of rectangular pulses modeling vehicle platoons.
\item Proposition~\ref{prop:existence-traffic-dynamics-solution} holds true also when the capacity function is state-dependent (referred to as \emph{adaptive} traffic signal control), but piecewise constant.
For example, let the capacity function $c_i(t)$ be equal to  $c_i^{\text{max}}$
if $t \in [\theta_i,\theta_i+g_i(0)] \cup [T+\theta_i, T+ \theta_i + g_i(1)] \cup \ldots$, and equal to zero otherwise, where $\theta_i \in [0,T]$ is the offset, and $\{g_i(0), g_i(1), \ldots\}$ is a sequence of green times. Such green times can be determined as a function of queue lengths. One such simple \emph{proportional} rule, when the capacity functions for all the incoming links at every intersection are mutually exclusive, is:
$$
g_i(k) = \frac{\|x_i(k-1:k)\|_{\infty}}{\sum_j \|x_j(k-1:k)\|_{\infty}}, \qquad k = 1, 2, \ldots
$$
The summation in the denominator is over all links incoming to the intersection to which $i$ is incident, and $\|x_i(k-1:k)\|_{\infty}:=\max_{t \in [(k-1)T,kT]} x_i(t)$ is the maximum queue length during the $k$-th cycle on link $i$.
\end{enumerate}
\end{remark}

\subsection{Periodic Solution}
\label{sec:periodic}
It is straightforward to see that the solution to \eqref{eq:flow-model-main} can be equivalently described in terms of $(x(t),z(t))$. Therefore, we shall use $(x(t),\beta(t))$ and $(x(t),z(t))$ interchangeably to refer to the solution to \eqref{eq:flow-model-main}.
We now develop a result analogous to the one in \cite{Muralidharan.Pedarsani.ea:15} on the existence of a globally attractive periodic orbit $(x^*(t),z^*(t))$, under the following  periodicity assumption.

\begin{assumption}
\label{ass:periodicity}
The external inflow functions $\{\lambda_i(t)\}_{i \in \mc E}$ and capacity functions $\{c_i(t)\}_{i \in \mc E}$ are all periodic with the same period $T>0$.\footnote{As noted in \cite{Muralidharan.Pedarsani.ea:15}, requiring the period to be the same is without loss of generality.}
\end{assumption}
Let 
\begin{equation}
\label{eq:average-quantity-notations}
\bar{\lambda}_i:=\frac{1}{T} \int_0^T \lambda_i(t) \, dt, \qquad \bar{c}_i:=\frac{1}{T} \int_0^T c_i(t) \, dt, \qquad i \in \mc E
\end{equation}
be the external inflow and capacity functions averaged over one period. Let $\bar{c} = \{\bar{c}_i: \, i \in \mc E\}$ and $\bar{\lambda}=\{\bar{\lambda}_i: \, i \in \mc E\}$ denote the collection of external inflow and capacity functions, respectively, for all links.
The following stability condition will be one of the sufficient conditions for establishing periodicity of $(x(t),z(t))$ at steady state.

\begin{definition}[Stability Condition]
\label{def:stability}
There exists $\epsilon>0$ such that $ [I-R^T] \bar{c} > \bar{\lambda} + \epsilon \onebf$.
\end{definition}

The proof of the following theorem is provided in the Appendix.

\begin{theorem}
\label{thm:globally-attractive-existence}
Let $\lambda(t)$, $c(t)$ and the initial condition $(x(0),\beta(0))$ satisfy Assumptions~\ref{ass:piece-wise-constant} and \ref{ass:periodicity}, and the stability condition in Definition~\ref{def:stability}. Then, there exists a unique periodic state trajectory $(x^*,z^*)$ with period $T$ for \eqref{eq:flow-model-main}, to which every trajectory converges.  
\end{theorem}

\subsection{Problem Statement}
While one can use \eqref{eq:flow-model-main} to obtain the steady state $(x^*,z^*)$ by direct simulations, in this paper, our objective is to develop an alternate framework to obtain $(x^*,z^*)$.

\section{Steady State Computation for an Isolated Link}
\label{sec:steady-state-single-link}
Let $y_i(t)$ be the cumulative inflow into link $i \in \mc E$. Referring to \eqref{eq:flow-dynamics}, this quantity is given by $y_i(t):=\lambda_i(t) + \sum_{j \in \mc E} R_{ji} z_j(t-\delta_{ji})$. For an isolated link $i$, $y_i(t)=\lambda_i(t)$. It is easy to see that $x_i^*(t) \equiv 0$ if $y_i(t) \leq c_i(t)$ for all $t \in [0,T)$. In order to avoid such trivialities, we assume that the set $\setdef{t \in [0,T)}{y_i(t) > c_i(t)}$ has non-zero measure.
The key in our approach is a procedure to easily compute $x_i^*(s)$ for some $s \in [0,T)$. Thereafter, $x_i^*(t)$ for all $t \in [0,T)$ can be easily obtained by simulating \eqref{eq:flow-model-main} over a time interval of length $T$. The natural candidates for such a $s \in [0,T)$ are the time instants when the queue length $x^*_i$ transitions between zero and positive values. We now provide a detailed procedure to compute such a transition point.
We implicitly assume throughout this and the next section that Assumption~\ref{ass:piece-wise-constant} and the stability condition in Definition~\ref{def:stability} holds true.


\begin{definition}[Transition Points]
\label{def:transition-points}
Let $\{\alpha_i^1, \ldots, \alpha_i^L\}$ be the time instants in $[0,T)$ when $x_i^*$ transitions from being zero to being positive. 
\end{definition}

Figure~\ref{fig:transition-points} illustrates the transition points for a sample scenario.

\begin{figure}[htb!]
\begin{center}				
\includegraphics[width=2.5in,angle=0]{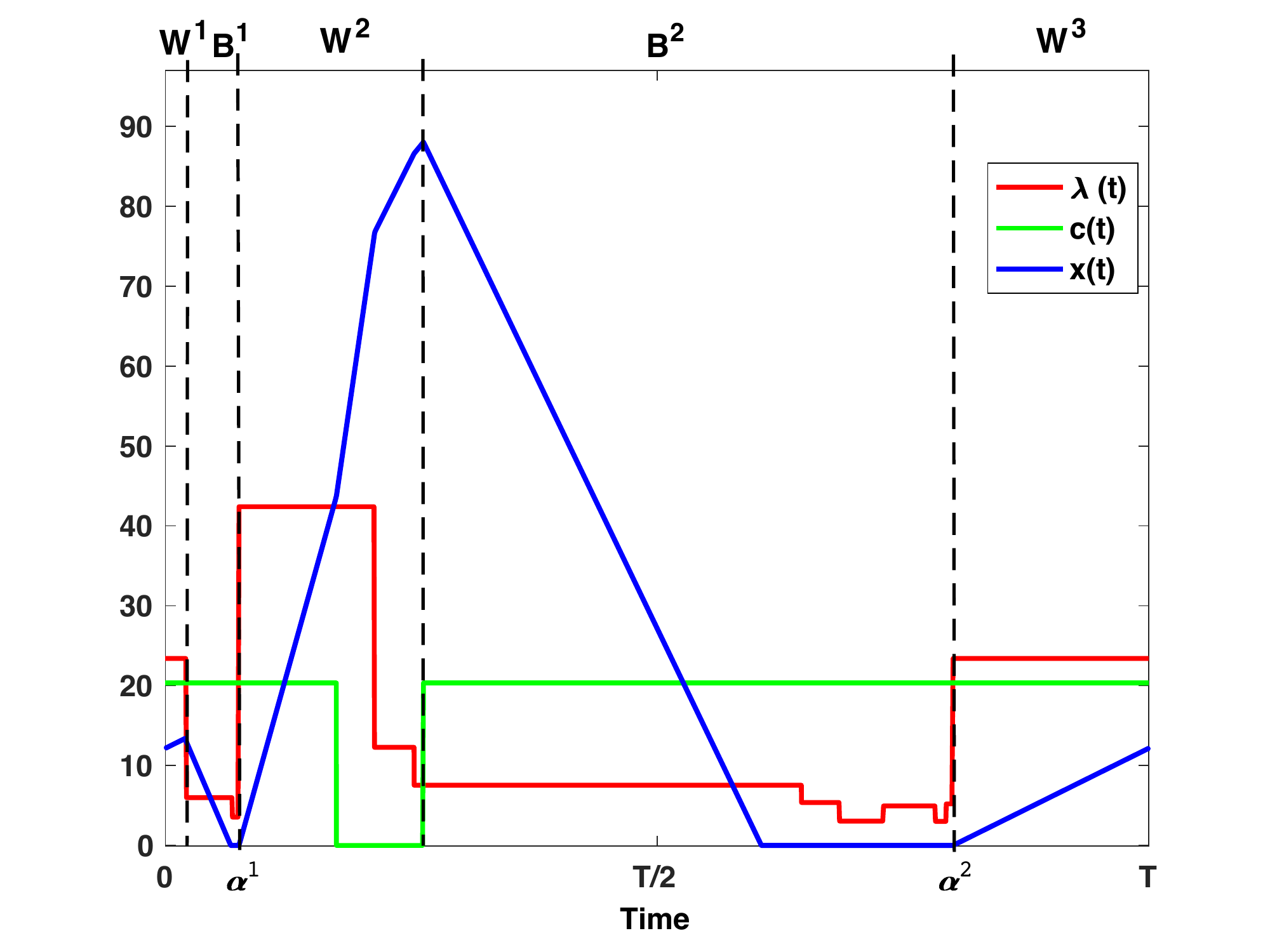} 
\end{center}
\caption{\small \sf Illustration of transition points, and negative/positive sets. In this case, $M^w=3$ and $M^b=2$. Subscript $i$ is not shown for brevity.}
\label{fig:transition-points}
\end{figure}

\begin{remark}
\label{rem:transition-point}
\begin{enumerate}
\item[(i)] Under the stability condition in Definition~\ref{def:stability}, $L \geq 1$, as also noted in \cite[Theorem 2]{Muralidharan.Pedarsani.ea:15}.
\item[(ii)] For a given $y_i(t)$ and $c_i(t)$, Theorem~\ref{thm:globally-attractive-existence} implies uniqueness of the resulting $(x_i^*,z_i^*)$, and hence of 
$\{\alpha_i^1, \ldots, \alpha_i^L\}$. 
%
%
\item[(iii)] As noted earlier, the knowledge of $x_i^*(s)$ at any single time instant $s \in [0,T)$ is sufficient to determine $x_i^*(t)$ over the entire period $[0,T)$. Indeed, $x^*(s)=0$ if $s \in \{\alpha_i^1, \ldots, \alpha_i^L\}$. By construction such a $x_i^*$ corresponds to a periodic orbit for \eqref{eq:flow-model-main}. 
Once $x_i^*$ is computed, inspired by Proposition~\ref{prop:solution-properties} and remarks immediately following it, let $z_i^*$ be given by:
\begin{equation}
\label{eq:z-star-def}
z_i^*(t) = \begin{cases} y_i(t) & x_i^*(t) = 0   \\
c_i(t) &  x_i^*(t) > 0
\end{cases}
\end{equation}
Periodicity of $x_i^*$, $y_i(t)$ and $c_i(t)$ imply that $z_i^*(t)$ in \eqref{eq:z-star-def} is periodic, i.e., $(x^*,z^*)$ is a periodic orbit. The uniqueness result in Theorem~\ref{thm:globally-attractive-existence} implies that this is indeed the desired object to be computed.
%
\end{enumerate}
\end{remark}

The time instant $s$ referenced in Remark~\ref{rem:transition-point} (iii), for whose computation we now provide a procedure, is $\alpha_i^L$. We need the notion of \emph{negative} and \emph{positive} sets, defined next.
\begin{definition}[Negative and Positive Sets]
\label{def:negative-positive-sets}
Let $\{B_i^1, \ldots, B_i^{M_i^b}\}$ be contiguous subsets of $[0,T)$ of non-zero size in which $y_i(t) < c_i(t)$, and let $\{W_i^1, \ldots, W_i^{M_i^w}\}$ be contiguous subsets of $[0,T)$ of non-zero size in which $y_i(t) > c_i(t)$. 
\end{definition}

\begin{remark}
\label{rem:positive-negative-sets}
\begin{enumerate}
\item[(i)] Since the set $\setdef{t \in [0,T)}{y_i(t) > c_i(t)}$ is assumed to have non-zero measure, under the stability condition in Definition~\ref{def:stability}, we have $M_i^b \geq 1$ and $M_i^w \geq 1$.
\item[(ii)] The sets $\{B_i^1, \ldots, B_i^{M_i^b}\}$ and $\{W_i^1, \ldots, W_i^{M_i^w}\}$ do not necessarily form a partition of $[0,T)$. Specifically, they exclude sets where $y_i(t)=c_i(t)$. 
\end{enumerate}
\end{remark}

Illustration of negative and positive sets are included in Figure~\ref{fig:transition-points}.
In preparation for the next result, let $\overline{B}_i^{k}=[\underline{b}_i^{k},\bar{b}_i^{k}]$, $k \in \until{M_i^b}$ and $\overline{W}_i^{k}=[\underline{w}_i^{k},\bar{w}_i^{k}]$, $k \in \until{M_i^w}$ be closures of $B_i^k$ and $W_i^k$, respectively. 

\begin{proposition}
\label{prop:main-computation}
Consider a link $i$ with inflow function $y_i(t)$ and capacity function $c_i(t)$, both periodic with period $T$. Let the transition points and positive/negative sets be given by Definitions~\ref{def:transition-points} and \ref{def:negative-positive-sets} respectively. Then, there exists a strictly increasing $\map{\myk_\alpha}{\until{L}}{\until{M_i^w}}$ such that $\alpha_i^{\ell} = \underline{w}_i^{\myk_{\alpha}(\ell)}$ for all $\ell \in \until{L}$.
%
\end{proposition}
\begin{proof} 
We drop the subscript $i$ for brevity in notation. 
The strictly increasing property of $\myk_{\alpha}$, if it exists, is straightforward; we provide a proof for existence. For a given $\ell \in \until{L-1}$, we let $\gamma^{\ell} \in (\alpha^{\ell},\alpha^{\ell+1})$ denote the time instant in between $\alpha^{\ell}$ and $\alpha^{\ell+1}$ when the queue length transitions from being positive to being zero. Similarly, we let $\gamma^L \in (\alpha^L,T)$ be the time instant in between $\alpha^L$ and $T$ when the queue length transitions from being positive to zero if it exists, or else we let $\gamma^L=T$. We also let $\gamma^0 \in (0,\alpha^1)$ be the time instant in between $0$ and $\alpha^1$ when the queue length transitions from being positive to zero if it exists, or else we let $\gamma^0=0$.

Assume, by contradiction, that there exists $\ell  \in \until{L}$  such that  ${\alpha}^{\ell} \notin \{\underline{w}^{1}, \ldots, \underline{w}^{M^w}\}$. Let $$a_1:=\max \setdef{a \in \until{M^w}}{\underline{w}^a < \alpha^{\ell}}$$ if it exists, and is equal to zero otherwise. Similarly, let $a_2:=\max \setdef{a \in \until{M^b}}{\underline{b}^a \leq \alpha^{\ell}}$, if it exists, and is equal to zero otherwise. Since $a_1$ and $a_2$ can not both be equal to zero, we have $a_1 \neq a_2$. Therefore, consider the following cases, where we use the convention that $\underline{w}^0=0=\underline{b}^0$: 
\begin{enumerate}
\item $\underline{w}^{a_1} < \underline{b}^{a_2}$: 
From the definition of $a_1$, we have (i) $\alpha^{\ell} \in \left[\underline{b}^{a_2}, \underline{w}^{a_1+1}\right)$ if $a_1 < M^w$, or (ii) $\alpha^{\ell} \in \left[\underline{b}^{a_2}, T\right]$ otherwise. In case (i), $\exists \, \epsilon > 0$ such that $\alpha^{\ell} + \epsilon < \min\{\underline{w}^{a_1+1},\gamma^{\ell}\}$, implying $y(t)-z(t)=y(t)-c(t) \leq 0$ for all 
$t \in [\alpha^{\ell},\alpha^{\ell}+\epsilon]$. Similar argument holds true for case (ii). Therefore, 
%
$ 
x(\alpha^{\ell} + \epsilon)= x(\alpha^{\ell} ) + \int_{\alpha^{\ell}}^{\alpha^{\ell}+\epsilon} (y(t)-c(t)) \, dt \leq x(\alpha^{\ell}) = 0
$
which is in contradiction to $x(\alpha^{\ell}+\epsilon)>0$, since $\alpha^{\ell}+\epsilon \in \left(\alpha^{\ell}, \gamma^{\ell} \right)$. 

\item $\underline{b}^{a_2} < \underline{w}^{a_1}$: The definitions of $a_1$ and $a_2$ imply that $\alpha^{\ell} \in \left(\underline{w}^{a_1}, \bar{w}^{a_1}\right]$. Therefore, $\exists \, \epsilon >0$ such that $\alpha^{\ell} - \epsilon > \max\{\underline{w}^{a_1}, \gamma^{\ell-1}\}$, which implies that $y(t)>c(t)$ for all $t \in [\alpha^{\ell}-\epsilon,\alpha^{\ell}]$. Therefore, 
$x(\alpha^{\ell}) =x(\alpha^{\ell}-\epsilon) + \int_{\alpha^{\ell}-\epsilon}^{\alpha^{\ell}} (y(t)-z(t)) \, dt  = \int_{\alpha^{\ell}-\epsilon}^{\alpha^{\ell}} (y(t)-z(t)) \, dt  > \int_{\alpha^{\ell}-\epsilon}^{\alpha^{\ell}} (c(t)-z(t)) \, dt \geq 0$, 
which contradicts $x(\alpha^{\ell})=0$.
\end{enumerate}
This establishes the proposition. 
\end{proof}


Proposition~\ref{prop:main-computation} narrows down our search for $\alpha_i^L$. We now sharpen this result to the point where it readily yields $\alpha_i^L$. In prepartion for this result, we need a few more definitions. For $s_1, s_2 \in [0,T]$, let
$$
C_i(s_1,s_2):=\int_{s_1}^{s_2} c_i(t) \, dt, \qquad 
Y_i(s_1,s_2):=\int_{s_1}^{s_2} y_i(t) \, dt
$$

Let $\mc W_i^{\alpha} := \left\{\underline{w}_i^{\mykk_1}, \ldots,  \underline{w}_i^{\mykk_m}\right \}$ be such that $\mykk_1=1$, and, for $j \in \{2, \ldots, m\}$, 
\begin{multline}
\label{eq:M-def}
\mykk_j = \argmin \Big \{ind \in \{\mykk_{j-1}+1, \ldots, M_i^w\} \, \Big | \, \exists \, p \in \until{M_i^b} \text{ s.t. } \bar{b}_i^p \in [\underline{w}_i^{ind-1},\underline{w}_i^{ind}]  \quad \& \\ Y_i(\underline{w}_i^{\mykk_{j-1}},\bar{b}_i^p) \leq C_i(\underline{w}_i^{\mykk_{j-1}},\bar{b}_i^p) \text{ if } \bar{b}_i^p < \underline{w}_i^{ind}, \text{ or } Y_i(\underline{w}_i^{\mykk_{j-1}},\bar{b}_i^p) < C_i(\underline{w}_i^{\mykk_{j-1}},\bar{b}_i^p) \text{ if } \bar{b}_i^p = \underline{w}_i^{ind}\Big \}
\end{multline}

where $m$ is implicitly defined by the value of $\mykk_j$ where the set over which $\argmin$ is taken in \eqref{eq:M-def} is empty. In words, \eqref{eq:M-def} implies that, for $j = 2, \ldots, m$, $\mykk_j$ is the index of the next positive set before which  there exists a negative set over which the solution to \eqref{eq:flow-model-main}, assuming $x(\underline{w}_i^{\mykk_{j-1}})=0$, hits zero. The ``or" in the second line of \eqref{eq:M-def} is to ensure that the time instant when the trajectory hits zero does not coincide with $\underline{w}^{ind}$, which is a candidate for $\alpha_i^{\ell}$ for some $\ell \in \until{L}$ (cf. Proposition~\ref{prop:gamma-1-computation-correctness}). See Figure~\ref{fig:alpha-construction} for an illustration.

\begin{figure}[htb!]
\begin{center}				
\includegraphics[width=2.5in,angle=0]{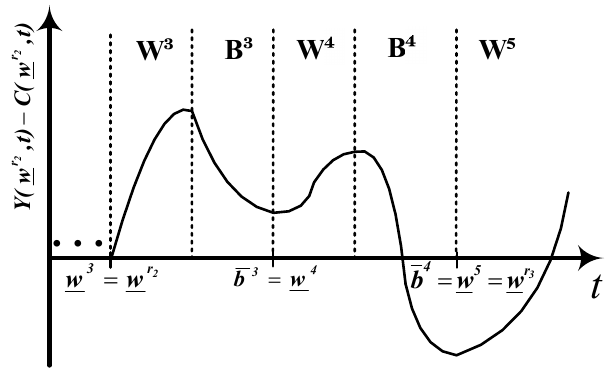} 
\end{center}
\caption{\small \sf Illustration of the procedure in \eqref{eq:M-def} to compute $\mc W_i^{\alpha}$. Specifically, the figure illustrates how \eqref{eq:M-def} determines $\underline{w}^{\mykk_3}$ to be equal to $\underline{w}^5$, given $\underline{w}^{\mykk_2} =\underline{w}^3$. Subscript $i$ is not shown for brevity.}
\label{fig:alpha-construction}
\end{figure}

Clearly, $\mc W_i^{\alpha} \subseteq \{\underline{w}_i^1, \ldots, \underline{w}_i^{M_i^w}\}$, which from Proposition~\ref{prop:main-computation} is known to contain $\{\alpha_i^1, \ldots, \alpha_i^L\}$. The next result shows that in fact the last $L$ entries of $\mc W_i^{\alpha}$ correspond to $\{\alpha_i^1, \ldots, \alpha_i^L\}$.


\begin{proposition}
\label{prop:gamma-1-computation-correctness}
Consider a link $i$ with inflow function $y_i(t)$ and capacity function $c_i(t)$, both periodic with period $T$, and the corresponding set $\mc W_i^{\alpha}$ defined via \eqref{eq:M-def}. Then $\{\alpha_i^1, \ldots, \alpha_i^L\} \subseteq \mc W_i^{\alpha}$, and, in particular, 
\begin{equation}
\label{eq:alpha-exact-expr}
\alpha_i^{\ell}=\underline{w}_i^{\mykk_m+\ell-L}, \qquad \ell \in \until{L}
\end{equation}
\end{proposition}
\begin{proof} 
%
We drop subscript $i$ for brevity in notation. Assume that there exists a $\ell \in \until{L}$ such that $\alpha^{\ell}=\underline{w}^{\myk_{\alpha}(\ell)} \notin \{\underline{w}^{\mykk_1}, \ldots, \underline{w}^{\mykk_m}\}$. let $\hat{w}$ be the largest element in $\{\underline{w}^{\mykk_1}, \ldots, \underline{w}^{\mykk_m}\}$ such that $\hat{w} < \alpha^{\ell}$. Since $\mykk_1=1$ (by definition), taking into account Proposition~\ref{prop:main-computation}, $\hat{w} \geq \underline{w}^{\mykk_1}$ is well-defined. Recall the definition of $\gamma^{\ell-1} \in (\alpha^{\ell-1},\alpha^{\ell})$ from the proof of Proposition~\ref{prop:main-computation}, and in particular that $\alpha^{\ell}$ is the $\underline{w}^k$ immediately after $\gamma^{\ell-1}$. 
If $\gamma^{\ell-1} < \hat{w}$, then $\alpha^{\ell} \leq \hat{w}$, giving a contradiction. Therefore, $\hat{w} < \gamma^{\ell-1} < \alpha^{\ell}$. It is easy to see that $\gamma^{\ell-1} \in (\underline{b}^{\zeta},\bar{b}^{\zeta}]$ for some $\zeta \in \until{M^b}$. Therefore, $x(\hat{w})+ Y(\hat{w},\bar{b}^{\zeta})-C(\hat{w},\bar{b}^{\zeta}) = x(\hat{w})+ Y(\hat{w},\gamma^{\ell-1})-C(\hat{w},\gamma^{\ell-1}) + Y(\gamma^{\ell-1},\bar{b}^{\zeta}) - C(\gamma^{\ell-1},\bar{b}^{\zeta}) = Y(\gamma^{\ell-1},\bar{b}^{\zeta}) - C(\gamma^{\ell-1},\bar{b}^{\zeta}) \leq 0$. This in turn would give $Y(\hat{w},\bar{b}^{\zeta})-C(\hat{w},\bar{b}^{\zeta}) \leq - x(\hat{w}) \leq 0$.\footnote{In order to minimize technicalities, we only cover the first case separated by ``or" in \eqref{eq:M-def}; when the second case holds, we would have $\gamma^{\ell-1} \in (\underline{b}^{\zeta}, \bar{b}^{\zeta})$ giving strict inequality.} However, referring to \eqref{eq:M-def}, this would imply $\myk_{\alpha}(\ell) \in \{\mykk_1, \ldots, \mykk_m\}$, giving a contradiction. This proves the first claim in the proposition. 

In order to prove \eqref{eq:alpha-exact-expr}, observe that if $\alpha^{\ell}=\underline{w}^{\mykk_j}$ for some $\ell \in \until{L-1}$, then \eqref{eq:M-def} implies $\alpha^{\ell+1}=\underline{w}^{\mykk_{j+1}}$. 
If we assume that $\myk_{\alpha}(L)=\mykk_{m_1}$ with  $m_1<m$, then \eqref{eq:M-def} implies that $\underline{w}^{\mykk_{m_1+1}} > \alpha^L$ is a point where the queue length transitions from being zero to being positive, giving a contradiction. Therefore, $\alpha^L=\underline{w}^{\mykk_m}$.
\end{proof}

Since $L$ is not known, \eqref{eq:alpha-exact-expr} can not be used to compute all $\alpha_i^{\ell}$, $\ell \in \until{L}$. However, \eqref{eq:alpha-exact-expr} readily gives
\begin{equation}
\label{eq:alphaL}
\alpha_i^{L}=\underline{w}_i^{\mykk_m}
\end{equation}
from which $x^*(t)$, $t \in [0,T]$, can then be computed as explained in Remark~\ref{rem:transition-point} (iii). In order to execute this last step, it is more convenient to use:
\begin{equation}
\label{eq:xzero-def}
x^*_i(0)=\left[Y_i(\alpha_i^L,T)-C_i(\alpha_i^L,T)\right]^+
\end{equation}
which is obtained by integrating \eqref{eq:flow-model-main} from $\alpha_i^L$ to $T$, and recalling \eqref{eq:M-def} for the definition of $\mykk_m$.
The entire procedure is summarized in Algorithm~\ref{alg:steady-state-computation-single-link}. 

\begin{algorithm}[htb!]
\Input {$T$- periodic inflow function $\lambda_i(t)$ and periodic capacity function $c_i(t)$}
\textbf{initialization}:  $y_i(t)=\lambda_i(t)$, $t \in [0,T]$\;
compute $\alpha_i^{L}$ from \eqref{eq:alphaL} and $x_i^*(0)$ from \eqref{eq:xzero-def}\;
compute $x_i^*(t)$, $t \in [0,T]$, by simulation of \eqref{eq:flow-model-main} with initial condition $x_i^*(0)$; compute $z_i^*(t)$, $t \in [0,T]$, from \eqref{eq:z-star-def}\; 
\caption{ Computation of $(x_i^*,z_i^*)$ for isolated link $i$}
\label{alg:steady-state-computation-single-link}
\end{algorithm}

Let the relationship between $(x_i^*,z_i^*)$ and $(y_i,c_i)$, as determined by Algorithm~\ref{alg:steady-state-computation-single-link}, be denoted by $x^*_i=\mc F_x(y_i,c_i)$ and $z^*_i=\mc F_z(y_i,c_i)$ respectively. These notations will be used in extending the procedure to compute steady state for the network. 

\section{Steady State Computation For a Network}
\label{sec:steady-state-network}
Algorithm~\ref{alg:steady-state-computation-network} formally describes the steps to compute steady-state for a general network. The number of iterations in the while loop in Algorithm~\ref{alg:steady-state-computation-network} is determined by a \emph{termination criterion}. While one could explicitly specify the number of iterations for termination criterion, a better criterion can be formulated as follows. For $i \in \mc E$, let $\bar{z}_i^*:=\frac{1}{T} \int_0^T z_i^*(t)$ be the average outflow from link $i$ at steady-state. Integrating \eqref{eq:flow-model-main} over $[0,T]$ at steady state, we get that $0=\bar{\lambda}=R^T \bar{z}^* - \bar{z}^*$, where we use notation from \eqref{eq:average-quantity-notations}. This then gives $\bar{z}^*=(I-R^T)^{-1} \bar{\lambda}$ (cf. Remark~\ref{ass:routing-matrix} for invertibility of $I-R^T$). Therefore, considering monotonicity of the iterates $z^{(k)}$ of Algorithm~\ref{alg:steady-state-computation-network} as established in Proposition~\ref{prop:outflow-update-monotonicity}, and letting $\bar{z}_i^{(k)}:=\frac{1}{T} \int_0^T z_i^{(k)}(t) \, dt$, a termination criterion could be $\max_{i \in \mc E} \left(\bar{z}^*_i-\bar{z}_i^{k} \right) \leq \epsilon$, for a specified $\epsilon>0$.

\begin{algorithm}[htb!]
\Input {periodic inflow functions $\lambda_i(t)$ and periodic capacity functions $c_i(t)$, $i \in \mc E$}
\textbf{initialization}:  $k=1$; $y_i^{(1)}(t)=\lambda_i(t)$, $t \in [0,T]$,  for all $i \in \mc E$\;
\While{termination criterion is not met}{
for all $i \in \mc E$: \\
\qquad compute $x_i^{(k)}=\mc F_x(y_i^{(k)},c_i)$ and $z_i^{(k)}=\mc F_z(y_i^{(k)},c_i)$ from Algorithm~\ref{alg:steady-state-computation-single-link}
\;
\qquad compute $y_i^{k+1}(t)=\lambda_i(t) + \sum_{j \in \mc E} R_{ji} z_i^{(k)}(t-\delta_{ji})$, $i \in \mc E$ \;
$k=k+1$\;
   }
\caption{ Computation of $(x_i^*,z_i^*)$, $i \in \mc E$}
\label{alg:steady-state-computation-network}
\end{algorithm}


\begin{proposition}
\label{prop:outflow-update-monotonicity}
Consider a network with $T$-periodic external inflows $\lambda(t)$ and $T$-periodic capacity functions $c(t)$. The link outflows computed by Algorithm~\ref{alg:steady-state-computation-network} satisfy the following for all $k$: $z^{(k+1)}(t) \geq z^{(k)}(t)$ and $x^{(k+1)}(t) \geq x^{(k)}(t)$ for all $t \in [0,T]$.
\end{proposition}
\begin{proof}
We prove by induction. Algorithm~\ref{alg:steady-state-computation-network} implies that, for all $i \in \mc E$, $y_i^{(2)}(t)=\lambda_i(t)+\sum_j R_{ji} z_j^{(1)}(t-\delta_{ji}) \geq \lambda_i(t) = y_i^{(1)}(t)$. Therefore, Corollary~\ref{cor:input-output-monotonicity} (in Appendix~\ref{sec:monotonicity}) implies that $x^{(2)}(t) \geq x^{(1)}(t)$ and $z^{(2)}(t) \geq z^{(1)}(t)$, and hence $z^{(2)}(t-\delta) \geq z^{(1)}(t-\delta)$ for all $i \in \mc E$, $t \in [0,T]$, $\delta \geq 0$. 

Assume that $z^{(k)}(t-\delta) \geq z^{(k-1)}(t-\delta)$ for all $k=2, \ldots, \bar{k}$. Since $y_i^{(\bar{k}+1)}(t)=\lambda_i(t)+\sum_j R_{ji} z_j^{(\bar{k})}(t-\delta_{ji}) \geq \lambda_i(t) +\sum_j R_{ji} z_j^{(\bar{k}-1)}(t-\delta_{ji}) = y_i^{(\bar{k})}(t)$, Corollary~\ref{cor:input-output-monotonicity} implies that $x_i^{(\bar{k}+1)}(t) \geq x_i^{(\bar{k})}(t)$ and $z_i^{(\bar{k}+1)}(t-\delta) \geq z_i^{(\bar{k})}(t-\delta)$ for all $i \in \mc E$, $t \in [0,T]$, $\delta \geq 0$. 
\end{proof}

$z_i^{(k)}(t) \leq c_i(t)$ for all $k$. An upper bound on $x^{(k)}$ can be shown along similar lines as Lemma~\ref{lem:boundedness} (in Appendix~\ref{sec:attractivity}). Combining this with monotonicity from Proposition~\ref{prop:outflow-update-monotonicity} implies that $(x^{(k)},z^{(k)})$ converges to $(\hat{x},\hat{z})$. Periodicity of $(x^{(k)},z^{(k)})$ for every $k$ implies periodicity of $(\hat{x},\hat{z})$.
It is easy to see from the construction of Algorithm~\ref{alg:steady-state-computation-network} that, for every iteration $k$:
$
\dot{x}^{(k)}_i(t) = \lambda_i(t) + \sum_{j \in \mc E} R_{ji} z_j^{(k-1)}(t-\delta_{ji}) - z_i^{(k)}(t)
$ 
for all $i \in \mc E$. Therefore, for any $t  \geq 0$:
$$
0 = x_i^{(k)}(t+T) - x_i^{(k)}(t) = \int_t^{t+T} \left(\lambda_i(s) + \sum_{j \in \mc E} R_{ji} z_j^{(k-1)}(s-\delta_{ji}) - z_i^{(k)}(s) \right) \, ds, \qquad i \in \mc E
$$
where the first equality follows from periodicity of $x^{(k)}$ by construction. Therefore, taking the limit as $k \to + \infty$, we get that, for all $t \geq 0$: 
$$
0 = \hat{x}_i(t+T) - \hat{x}_i(t) = \int_t^{t+T} \left(\lambda_i(s) + \sum_{j \in \mc E} R_{ji} \hat{z}_j(s-\delta_{ji}) - \hat{z}_i(s) \right) \, ds, \qquad i \in \mc E
$$
This implies that $(\hat{x},\hat{z})$ is a periodic orbit for \eqref{eq:flow-dynamics}. The uniqueness result in Theorem~\ref{thm:globally-attractive-existence} then implies that it is indeed the object to be computed. 

\begin{remark}
Algorithm~\ref{alg:steady-state-computation-network} naturally lends itself to a distributed implementation: during an iteration, all the links independently update their respective $(x_i^{(k)},z_i^{(k)})$, and at the end of the iteration, each link transmits its updated $z_i^{(k)}$ to its immediately downstream links.
\end{remark}

\section{Simulations}
\label{sec:simulations}

In this section, we report simulation results in two parts. In Section~\ref{subsec:matlab}, we illustrate consistency between teady-state computations from Algorithm~\ref{alg:steady-state-computation-network} and the steady-state obtained from direct simulations in MATLAB, on a synthetic network. In Section~\ref{subsec:vissim}, we report comparison between steady-state computations from Algorithm~\ref{alg:steady-state-computation-network} with the output from PTV VISSIM, a well-known microscopic traffic simulator, for a sub-network in downtown Los Angeles. 

\subsection{MATLAB simulations}
\label{subsec:matlab}
\begin{figure}[htb!]
\begin{center}
\includegraphics[width=1.95 in]{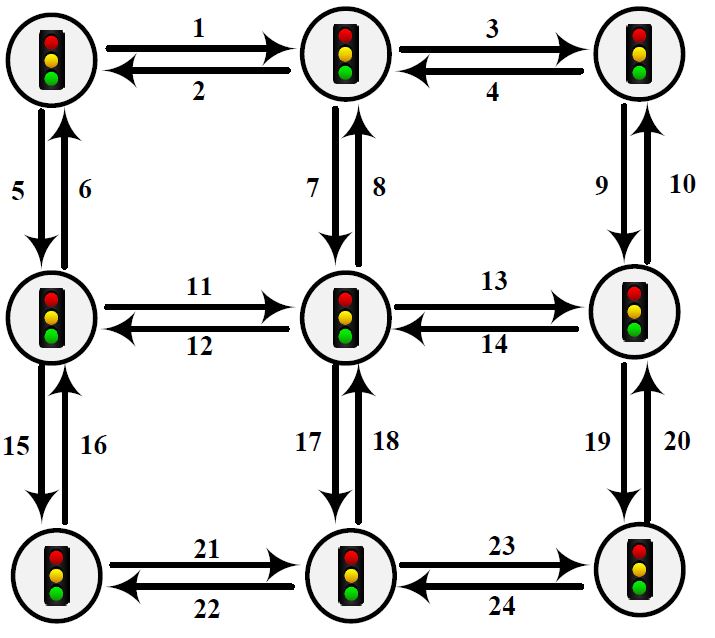}
\caption{\small \sf Graph topology of the network used in the simulations.}
\vspace{-0.3in}
\label{fig-network}
\end{center}
\end{figure}
The graph topology of the network used for simulations is shown in Figure~\ref{fig-network}. All intersections have common cycle time of $T=20$. The external inflows are constant $\lambda_i(t) \equiv \lambda_i$, $i \in \mc E$, and the values are provided in Table \ref{table:inflow}. The capacity functions are of the form: $c_i(t)=c_i^{\text{max}}$ if $t \in [\theta_i,\theta_i+g_i]$ and zero otherwise. The values of $c_i^{\text{max}}$, $\theta_i$ and $g_i$, which can be interpreted as saturation capacity, offset and green split are given in Table \ref{table:capacity}.

\begin{table}[htb!]
\centering
\begin{tabular}{|c|c|c|c|c|c|c|c|c|c|c|c|c|l|} \hline
Link ID ($i$)& 1 & 2 & 3 & 4& 5 & 6 & 7 & 8 & 9 & 10& 11 & 12\\ \hline
$\lambda_i$& 1.70  &  2.27  & 4.35   & 3.11   & 9.23    &  4.30 & 1.84 &  9.04  & 9.79  & 4.38& 1.11   &  2.58 \\ \hline   \hline
Link ID ($i$)& 13 & 14 & 15 & 16& 17 & 18 & 19 & 20 & 21& 22 & 23 & 24\\ \hline
$\lambda_i$& 4.08  &5.94& 2.62&  6.02 & 7.11 & 2.21 &1.17 &2.96 & 3.18  &4.24  & 5.07 &  0.85 \\ \hline   
\end{tabular}
\caption{\small \sf External inflow values.}
\label{table:inflow}
\end{table}


\begin{table}[htb!]
\centering
\begin{tabular}{|c|c|c|c|c|c|c|c|c|c|c|c|c|l|} \hline
Link ID ($i$)& 1 & 2 & 3 & 4& 5 & 6 & 7 & 8 & 9 & 10& 11 & 12\\ \hline
$c_i^{\text{max}}$&47.81 & 20.34 & 147.74 & 212.15& 363.33 &1192.03 &362.82 &67.30& 706.05 & 69.93& 142.51 &114.89\\ \hline
$\theta_i$& 16.02&5.24&0.58&1.49&18.57&3.47&14.60&5.05&9.77&7.01&15.88&4.72 \\ \hline
$g_i$& 5.47&18.22&6.42&3.56&6.15&1.77&1.27&10.96&1.78&13.57&8.14&8.92  \\ \hline   \hline

Link ID ($i$)& 13 & 14 & 15 & 16& 17 & 18 & 19 & 20 & 21& 22 & 23 & 24\\ \hline
$c_i^{\text{max}}$&48.06&154.75&134.76&279.76&98.60&131.25&94.35&98.10&398.25&94.87&107.44&176.15 \\ \hline
$\theta_i$& 11.56&4.02&11.57&13.65&4.74&9.45&9.17&2.09&15.94&19.69&17.20&18.05\\ \hline
$g_i$& 10.53&5.43&8.12&4.92&11.19&5.15&8.02&7.67&2.11&11.87&11.97&6.68 \\ \hline  
 
\end{tabular}
\caption{\small \sf Parameters of link capacity functions.}
\label{table:capacity}
\end{table}

These values of $\lambda(t)$ and $c(t)$ satisfy the stability condition in Definition~\ref{def:stability}. 
The routing matrix is chosen to be: 
\newcommand\scalemath[2]{\scalebox{#1}{\mbox{\ensuremath{\displaystyle #2}}}}
\[
R  =  \left(
    \scalemath{0.45}{
    \begin{array}{cccccccccccccccccccccccccccccccc}
     0 & 0.44 & 0.23 & 0 & 0 & 0   &  0.23 & 0 & 0 & 0 & 0 & 0     &  0 & 0 & 0 & 0 & 0 & 0     &  0 & 0 & 0 & 0 & 0 & 0   \\
     0.33& 0 & 0 & 0 & 0.57 & 0   &  0 & 0 & 0 & 0 & 0 & 0     &  0 & 0 & 0 & 0 & 0 & 0     &  0 & 0 & 0 & 0 & 0 & 0 \\
      0 & 0 & 0 & 0.34 & 0 & 0   &  0 & 0 & 0.56 & 0 & 0 & 0     &  0 & 0 & 0 & 0 & 0 & 0     &  0 & 0 & 0 & 0 & 0 & 0  \\
      0 & 0.19 & 0.5 & 0 & 0 & 0   &  0 & 0.21 & 0 & 0 & 0 & 0     &  0 & 0 & 0 & 0 & 0 & 0     &  0 & 0 & 0 & 0 & 0 & 0  \\
      
           0 & 0 & 0 & 0 & 0 & 0.34   &  0 & 0 & 0 & 0 & 0.35 & 0     &  0 & 0 & 0.21 & 0 & 0 & 0     &  0 & 0 & 0 & 0 & 0 & 0   \\
     0.05 & 0 & 0 & 0 & 0.85 & 0   &  0 & 0 & 0 & 0 & 0 & 0     &  0 & 0 & 0 & 0 & 0 & 0     &  0 & 0 & 0 & 0 & 0 & 0 \\
      0 & 0 & 0 & 0 & 0 & 0   &  0 & 0.06 & 0 & 0 & 0 & 0.27     &  0.3 & 0 & 0 & 0 & 0.27 & 0     &  0 & 0 & 0 & 0 & 0 & 0  \\
      0 & 0.08 & 0.55 & 0 & 0 & 0   &  0.27 & 0 & 0 & 0 & 0 & 0     &  0 & 0 & 0 & 0 & 0 & 0     &  0 & 0 & 0 & 0 & 0 & 0  \\

           0 & 0 & 0 & 0 & 0 & 0   &  0 & 0 & 0 & 0.41 & 0 & 0     &  0 & 0.26 & 0 & 0 & 0 & 0     &  0.23 & 0 & 0 & 0 & 0 & 0   \\
     0	&0	&0&	0.38&	0&	0&	0&	0&	0.52&	0&	0&	0&	0&	0&	0&	0&	0&	0&	0&	0	&0	&0&	0&	0 \\
      0	&0&	0&	0&	0&	0&	0&	0.2&	0	&0	&0	&0.24&	0.13&	0&	0&	0&	0.33 &	0&	0&	0&	0&	0&	0&	0  \\
0&	0&	0&	0&	0&	0.08&	0&	0&	0&	0&	0.33&	0&	0&	0&	0.49&	0&	0&	0&	0&	0&	0&	0&	0&	0 \\

          0&	0&	0&	0&	0&	0&	0&	0&	0&	0.32&	0&	0&	0&	0.22&	0&	0&	0&	0&	0.36&	0&	0&	0&	0&	0   \\
    0&	0&	0&	0&	0&	0&	0&	0.12&	0&	0&	0&	0.43&	0.12&	0&	0&	0&	0.23&	0&	0&	0&	0&	0&	0&	0 \\
    0&	0&	0&	0&	0&	0&	0&	0&	0&	0&	0&	0&	0&	0&	0&	0.57&	0&	0&	0&	0&	0.33&	0&	0&	0  \\
    0&	0&	0&	0&	0&	0.84&	0&	0&	0&	0&	0.02&	0&	0&	0&	0.04&	0&	0&	0&	0&	0&	0&	0&	0&	0 \\
      
     0&	0&	0&	0&	0&	0&	0&	0&	0&	0&	0&	0&	0&	0&	0&	0&	0&	0.15&	0&	0&	0&	0.39&	0.36&	0   \\
   0&	0&	0&	0&	0&	0&	0&	0.23&	0&	0&	0&	0.29&	0.03&	0&	0&	0&	0.34&	0&	0&	0&	0&	0&	0&	0 \\
      0&	0&	0&	0&	0&	0&	0&	0&	0&	0&	0&	0&	0&	0&	0&	0&	0&	0&	0&	0.19&	0&	0&	0&	0.71   \\
    0&	0&	0&	0&	0&	0&	0&	0&	0&	0.22&	0&	0&	0&	0.35&	0&	0&	0&	0&	0.33&	0&	0&	0&	0&	0  \\
      
           0&	0&	0&	0&	0&	0&	0&	0&	0&	0&	0&	0&	0&	0&	0&	0&	0&	0.28&	0&	0&	0&	0.32&	0.3&	0   \\
     0&	0&	0&	0&	0&	0&	0&	0&	0&	0&	0&	0&	0&	0&	0&	0.54&	0&	0&	0&	0&	0.36&	0&	0&	0 \\
      0&	0&	0&	0&	0&	0&	0&	0&	0&	0&	0&	0&	0&	0&	0&	0&	0&	0&	0&	0.42&	0&	0&	0&	0.48   \\
     0&	0&	0&	0&	0&	0&	0&	0&	0&	0&	0&	0&	0&	0&	0&	0&	0&	0.19&	0&	0&	0&	0.28&	0.43&	0\\
      
    \end{array}
    }
  \right)
\]

While the entries of $R$ are chosen arbitrarily, the sum of entries on each row is $0.9<1$. This combined with the fact that the network shown in Figure~\ref{fig-network} is weakly connected, Assumption~\ref{ass:connectivity} is satisfied. The inter-link travel times are all taken to be zero, i.e., $\bar{\delta}=0$. 

Figure~\ref{fig-rmse} shows the root mean squared error (RMSE) between the $(x^*,z^*)$ obtained by direct simulation of \eqref{eq:flow-model-main}, with initial condition $x_i(0)=10$ for all $i \in \mc E$, over a sufficiently long time horizon, and the output $(x^{(k)},z^{(k)})$ of Algorithm~\ref{alg:steady-state-computation-network} during various iterations. The RMSE between $x_i^{(k)}$ and $x_i^*$ is defined as $\sqrt  {\frac{ \int_{t=0}^{T} { (x_i^{(k)}(t)-x^*_i(t))^2 } \, dt}{ T } }$. The RMSE between $z_i^{(k)}$ and $z_i^*$ is defined similarly. The monotonically decreasing RMSE in Figure~\ref{fig-rmse} (a) and (b) illustrates the monotonic convergence of the iterates of Algorithm~\ref{alg:steady-state-computation-network} to the desired periodic orbit $(x^*,z^*)$. Figure~\ref{fig-rmse} (c) illustrates uniform monotone convergence of $x_i^{(k)}$ to $x_i^*$, as stated in Proposition~\ref{prop:outflow-update-monotonicity},  for a sample link.

\begin{figure}
\begin{center}

\includegraphics[width=1.9 in]{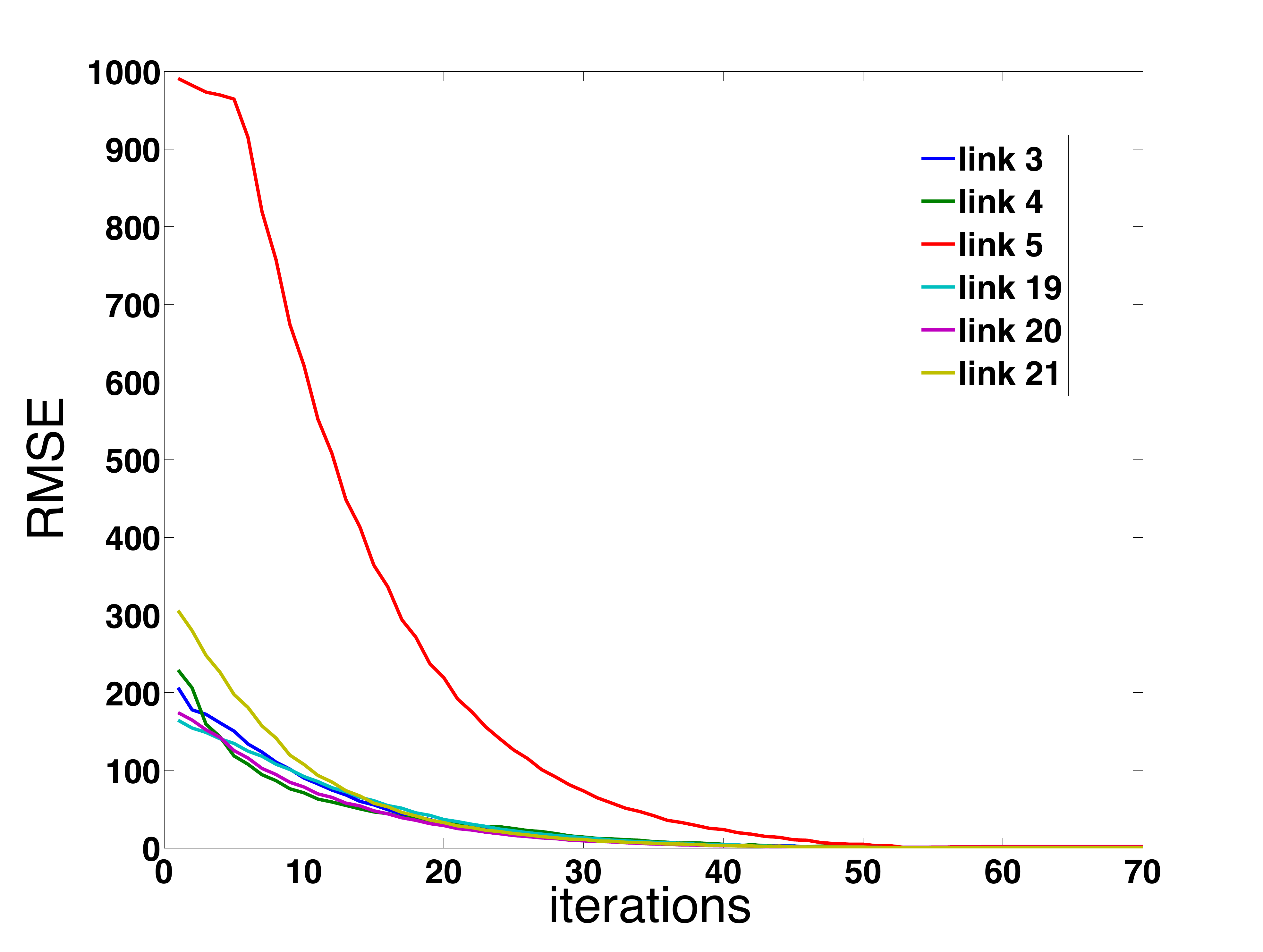} 
\hspace{0.1in}
\includegraphics[width=1.83 in]{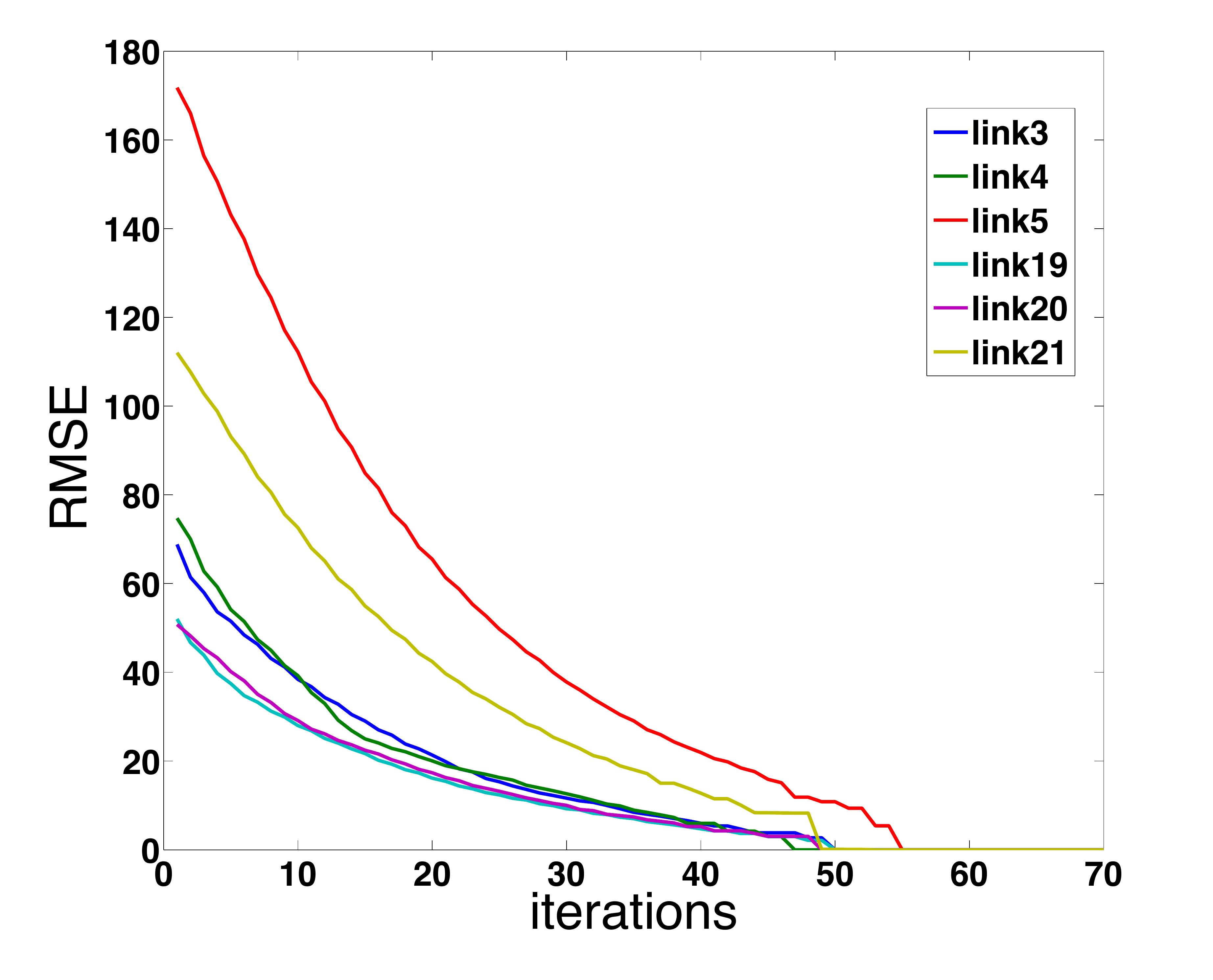} 
\hspace{0.1in}
\includegraphics[width=2.39 in]{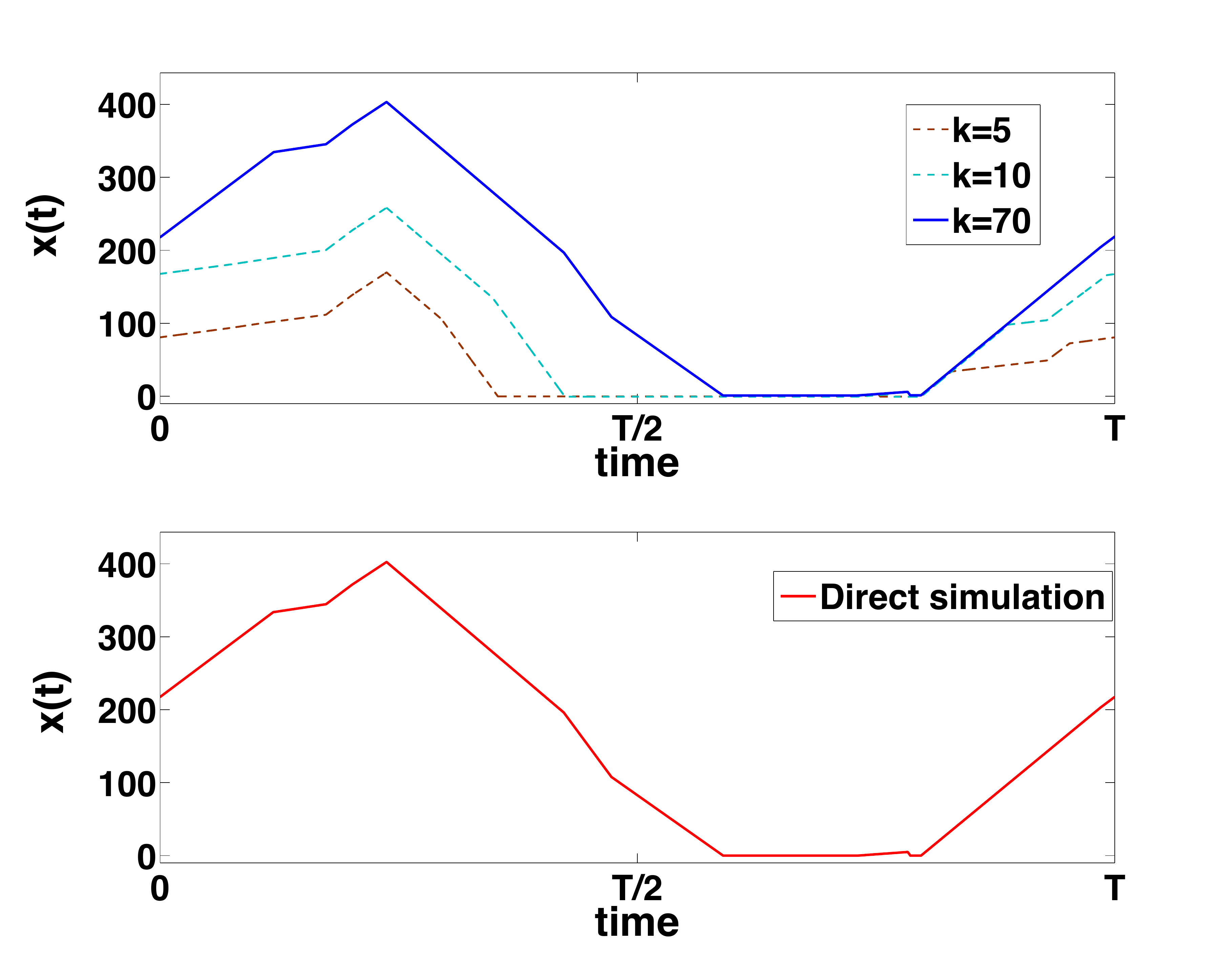} \\
(a) \hspace{1.85in}
(b) \hspace{1.85in}
(c)
\caption{\small \sf  Evolution of RMSE between (a) $x^{(k)}$ and $x^*$, and (b) $z^{(k)}$ and $z^*$ for a few representative links. (c) (top) $x_i^{(k)}$ from a few representative iterations and (bottom) $x_i^*$, both for $i=17$.
%
}
\vspace{-0.3in}
\label{fig-rmse}
\end{center}
\end{figure}



\subsection{VISSIM simulations}
\label{subsec:vissim}


\begin{figure}
\begin{center}
\includegraphics[width=2.45 in]{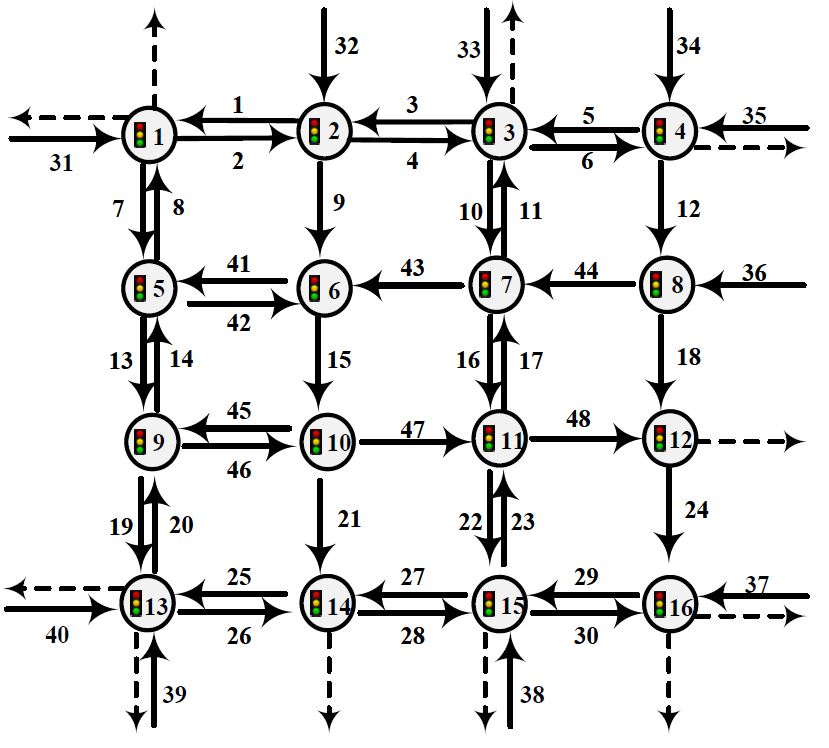}
\hspace{0.1in}
\includegraphics[width=2.05 in]{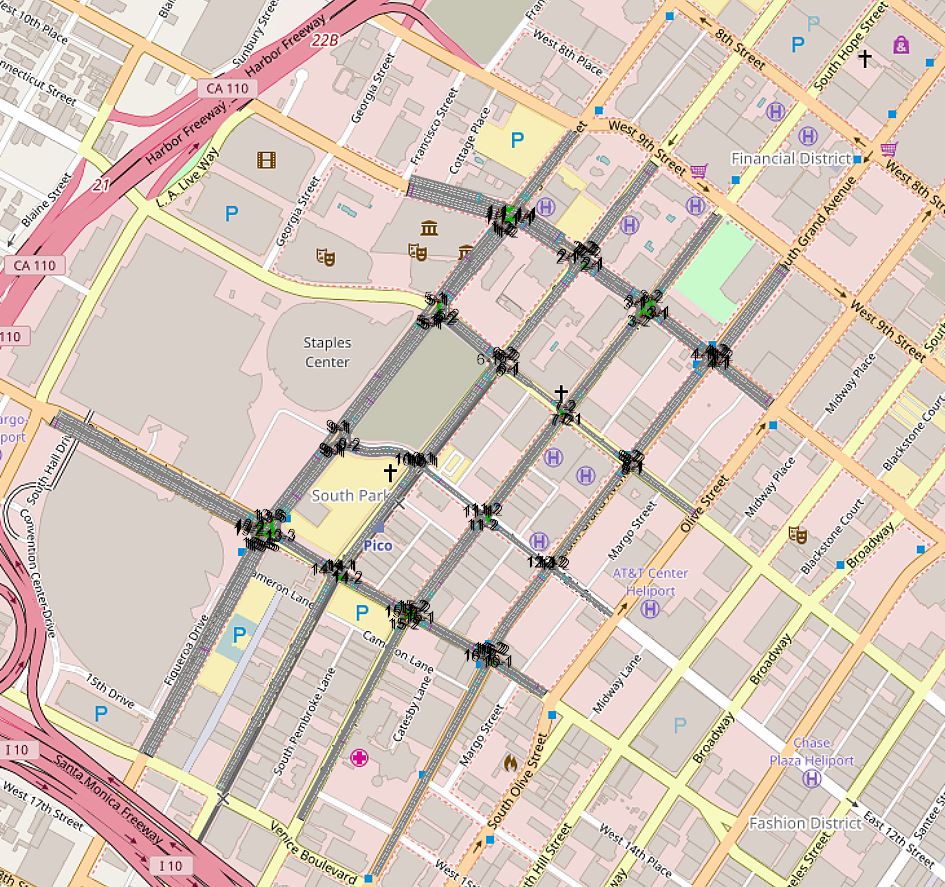} \\
\hspace{1.49 in}
(a) \hspace{2.1in}
(b) \hspace{1.85in}
\caption{\small \sf  The Los Angeles downtown sub-network used in the simulations: (a) graph topology
(b) map view.
%
}
\vspace{-0.3in}
\label{fig-LA-network}
\end{center}
\end{figure}


In this section, we report comparison between steady-state computations from Algorithm~\ref{alg:steady-state-computation-network} with the output from PTV VISSIM for the downtown Los Angeles sub-network shown in Figure~\ref{fig-LA-network}.

All the intersections have common cycle time of $T=90$ second. Referring to the notations in Section~\ref{sec:problem-formulation}, the values of offsets and green times for various capacity functions were obtained from Los Angeles Department of Transportation (LADOT) signal timing sheets, and are reported in Table~\ref{table:la-capacity}. 

The values of saturated flow capacities, denoted as $c^{\text{max}}$ in Table~\ref{table:la-capacity}, are based on the values commonly reported in the literature, e.g., \cite{Papacostas.Prevedouros:01}: 1800 vehicle/hour/lane for through movement, 1600 vehicle/hour/lane for right-turn and left-turn movements, 1200 vehicle/hour/lane for lanes that are shared between through movement and left/right turn and 960 vehicle/hour/lane for permissive left turns. For each link, the total saturated flow capacity is computed by adding up the capacities of all the lanes associated with it; see Figure~\ref{fig:movements} for an illustration. The resulting values for all the links are reported in Table~\ref{table:la-capacity}.
The external inflows $\lambda(t) \equiv \lambda$ are taken to be non-zero only on the boundary links. These values, which are estimated from loop detector data during weekday PM peak hour (between 4pm to 6pm) from May 1 to May 31, 2013, are reported in Table \ref{table:la-inflow}.
%



\begin{figure}[htb!]
\begin{minipage}[c]{.99\textwidth}
\begin{center}
\includegraphics[width=2.35 in]{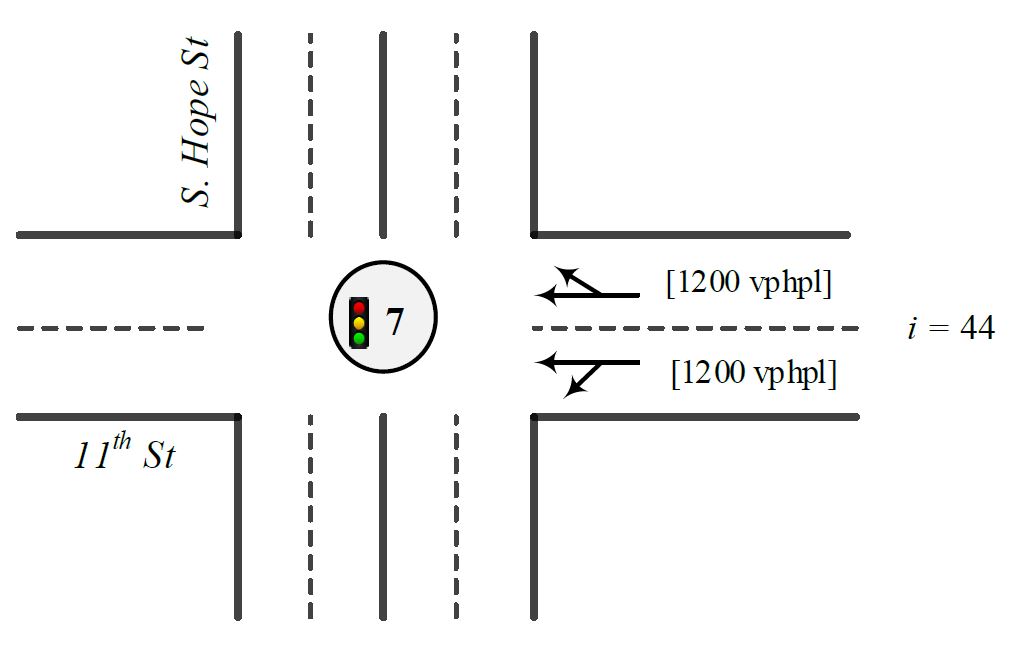} 
\end{center}
\end{minipage}
\caption{\small \sf Illustration of \emph{movements}, and \emph{lanes} on link 44 at intersection number 7 in the network shown in Figure~\ref{fig-LA-network}. Link number 44 contains two lanes: one lane supports through+right movements, and the second lane supports through+left movements. Therefore, the saturation capacity of link 44 is $c_4^{\text{max}}=1200+1200=2400 \text{ veh/hour}$, as also noted in Table~\ref{table:la-capacity}.} 
\label{fig:movements}
\end{figure}

For every link, the turn ratios are chosen to be proportional to the number of lanes dedicated to each movement. For example, for link 44 (cf. Figure~\ref{fig:movements}),  the ratios are 0.5, 0.25 and 0.25 for through movement, right turn and left turn, respectively.
As a result, the sum of entries of rows of $R$ associated with links which have downstream exit links, shown in dashed arrow in Figure~\ref{fig-LA-network}, is strictly less than one. An example is link 8. On the other hand, for links with no downstream exit links, e.g., link 44, the entries of the corresponding row in $R$ add up to be equal to one.
Combining this with the fact that the network shown in Figure~\ref{fig-LA-network} (a) is weakly connected, Assumption~\ref{ass:connectivity} is satisfied in this case.
Moreover, matrix $R$ and values in Tables~\ref{table:la-capacity} and \ref{table:la-inflow} satisfy the stability condition in Definition~\ref{def:stability}. 
The link travel times $\bar{\delta_i}$, $i \in \mc E$, are constant and equal to free-flow travel time, i.e., the link length divided by the link speed limit. These values are presented in Table \ref{table:la-travel-time}.  

\begin{table}[htb!]
\centering
\begin{tabular}{|c|c|c|c|c|c|c|c|c|c|c|c|c|l|} \hline
Link ID ($i$)& 1 & 2 & 3 & 4& 5 & 6 & 7 & 8 & 9 & 10& 11 & 12\\ \hline
$c_i^{\text{max}}$ (veh/hour)&7860&5100&4560&3960&3960&3000&4560&8500&6600&3000&2400&7000\\ \hline
$\theta_i$ (sec)&88&77&77&73&73&55&38&50&1&31&31&50 \\ \hline
$g_i$ (sec)&52&43&43&48&48&48&53&36&44&44&40&49  \\ \hline   \hline

Link ID ($i$)& 13 & 14 & 15 & 16& 17 & 18 & 19 & 20 & 21& 22 & 23 & 24\\ \hline
$c_i^{\text{max}}$ (veh/hour)& 4560&6600&6400&3000&3000&4800&3360&6600&4200&4260&3000&6300  \\ \hline
$\theta_i$ (sec)&16&38&65&63&31&64&30&16&64&87&63&84\\ \hline
$g_i$ (sec)&51&53&59&52&44&54&44&51&44&39&52&39 \\ \hline  \hline

Link ID ($i$)& 25& 26&27&28&29&30&31&32&33&34&35&36\\ \hline
$c_i^{\text{max}}$ (veh/hour)&2400&3000&3000&2400&2400&3000&7860&6600&2400&8500&4560&5200\\ \hline
$\theta_i$ (sec)&76&18&18&36&36&35&88&30&31&13&55&11 \\ \hline
$g_i$ (sec)&44&45&45&51&51&49&52&45&40&40&48&39 \\ \hline   \hline

Link ID ($i$)& 37 &38 & 39& 40&41 &42&43&44&45&46&47&48\\ \hline
$c_i^{\text{max}}$ (veh/hour)&3000&4260&7560&6060&6200&1500&5200&2400&3100&3000&2400&3000\\ \hline
$\theta_i$ (sec)&35&87&30&76&1&47&47&77&67&34&27&30 \\ \hline
$g_i$ (sec)&49&39&44&44&36&44&44&44&37&29&36&34 \\ \hline  
 
\end{tabular}
\caption{\small \sf Parameters of link capacity functions.}
\label{table:la-capacity}
\end{table}

\begin{table}[htb!]
\centering
\begin{tabular}{|c|c|c|c|c|c|c|c|c|c|c|l|} \hline
Link ID ($i$)& 31 & 32 & 33 & 34& 35 & 36 & 37 & 38 & 39 & 40\\ \hline
$\lambda_i$ (veh/hour)& 1271.4  &  270.4  & 755.3   & 573.3   & 414.7   &  694.2 & 185.9 &  185.9 & 1323.4   & 826.8 \\ \hline  
\end{tabular}
\caption{\small \sf External inflow on links located on the boundary of the network.}
\label{table:la-inflow}
\end{table}


\begin{table}[htb!]
\resizebox{\textwidth}{!}{

\centering
\begin{tabular}{|c|c|c|c|c|c|c|c|c|c|c|c|c|c|c|c|c|c|c|c|c|c|c|c|c|l|} \hline
Link ID ($i$)& 1 & 2 & 3 & 4& 5 & 6 & 7 & 8 & 9 & 10& 11 & 12 & 13 & 14 & 15 & 16& 17 & 18 & 19 & 20 & 21& 22 & 23 & 24 \\ \hline
$\bar{\delta_i}$ (sec)& 8&8&8&8&8&8&9&9&10&11&11&11&16&16&11&11&11&11&11&11&12&10&10&9 \\ \hline   \hline

Link ID ($i$)& 25& 26&27&28&29&30&31&32&33&34&35&36 & 37 &38 & 39& 40&41 &42&43&44&45&46&47&48\ \\ \hline
$\bar{\delta_i}$ (sec)& 9&9&8&8&9&9&9&10&10&10&10&8&5&22&23&23&7&7&8&8&7&7&8&8 \\ \hline   

\end{tabular}
}
\caption{\small \sf Link travel times.}
\label{table:la-travel-time}
\end{table}


Figure~\ref{fig-LA-rmse} shows RMSE between the $(x^*,z^*)$ obtained by direct simulation of \eqref{eq:flow-model-main}, with initial condition $x_i(0)=10$ for all $i \in \mc E$, over a sufficiently long time horizon, and the output $(x^{(k)},z^{(k)})$ of Algorithm~\ref{alg:steady-state-computation-network} during various iterations. 
The monotonically decreasing RMSE in Figure~\ref{fig-LA-rmse} (a) and (b) illustrates the monotonic convergence of the iterates of Algorithm~\ref{alg:steady-state-computation-network} to the desired periodic orbit $(x^*,z^*)$. Figure~\ref{fig-LA-rmse} (c) illustrates uniform monotone convergence of $x_i^{(k)}$ to $x_i^*$, as stated in Proposition~\ref{prop:outflow-update-monotonicity},  for a sample link. 

\begin{figure}
\begin{center}
\includegraphics[width=1.9 in]{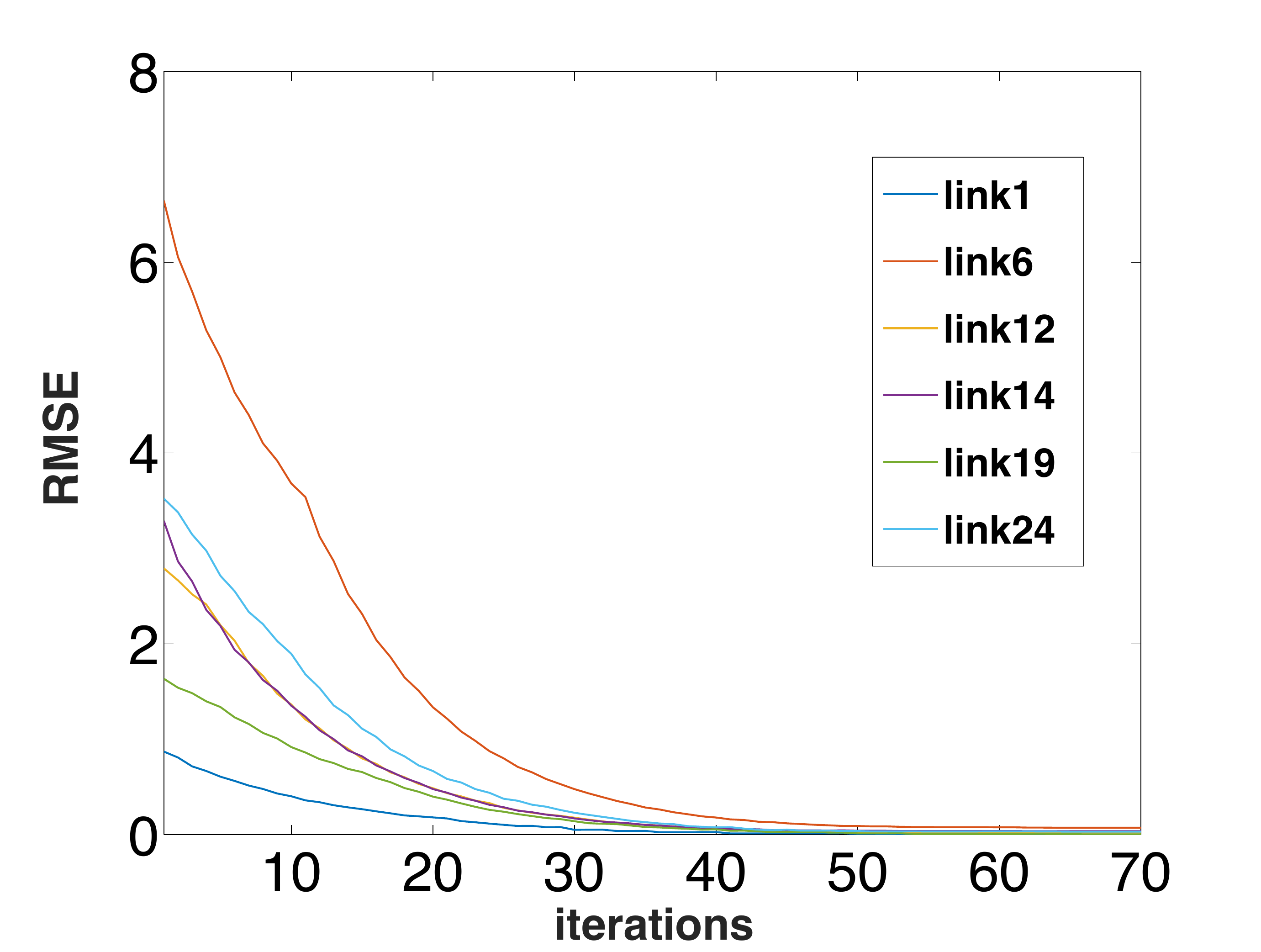} 
\hspace{0.1in}
\includegraphics[width=1.985 in]{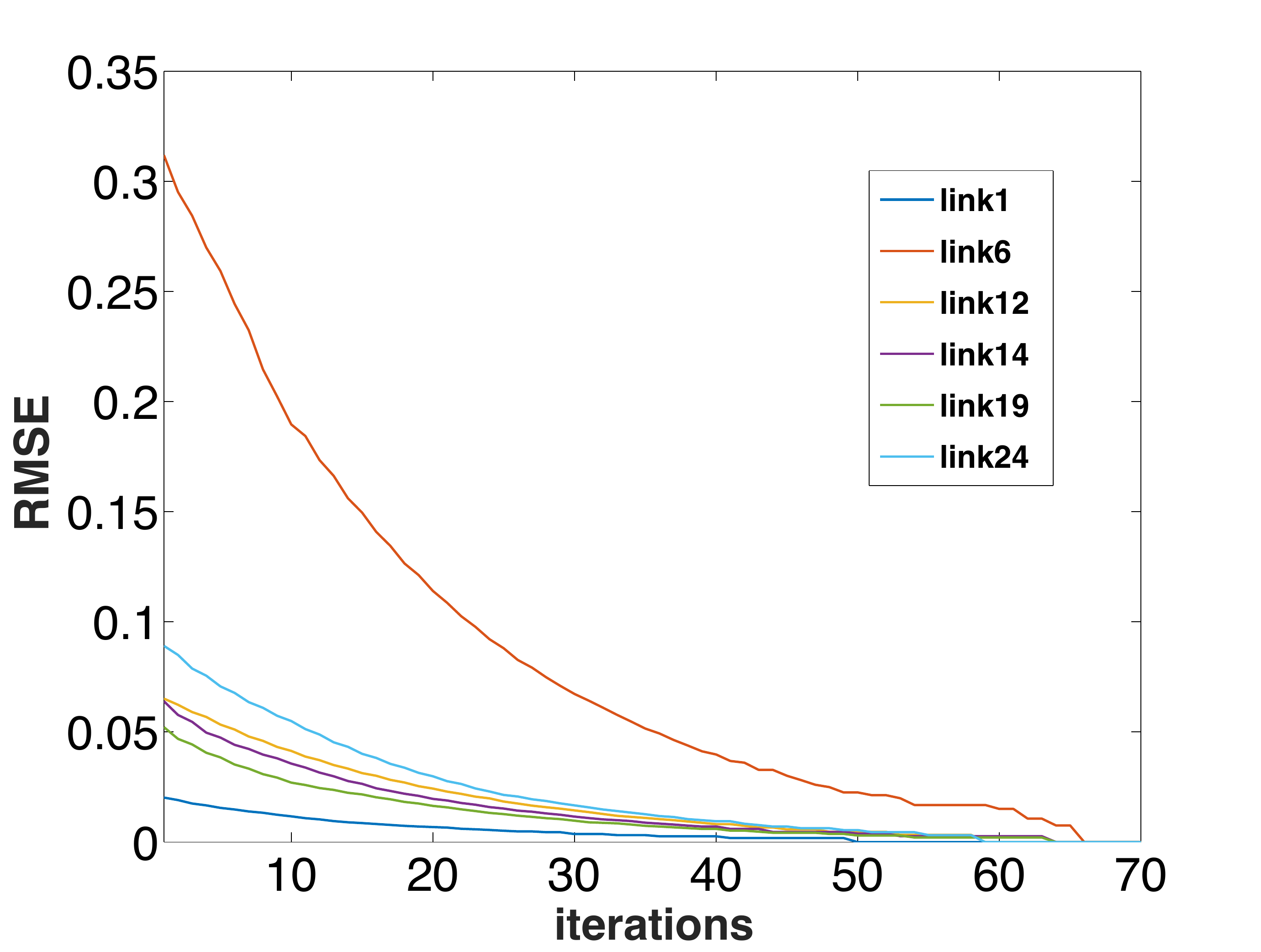} 
\hspace{0.1in}
\includegraphics[width=2.09 in]{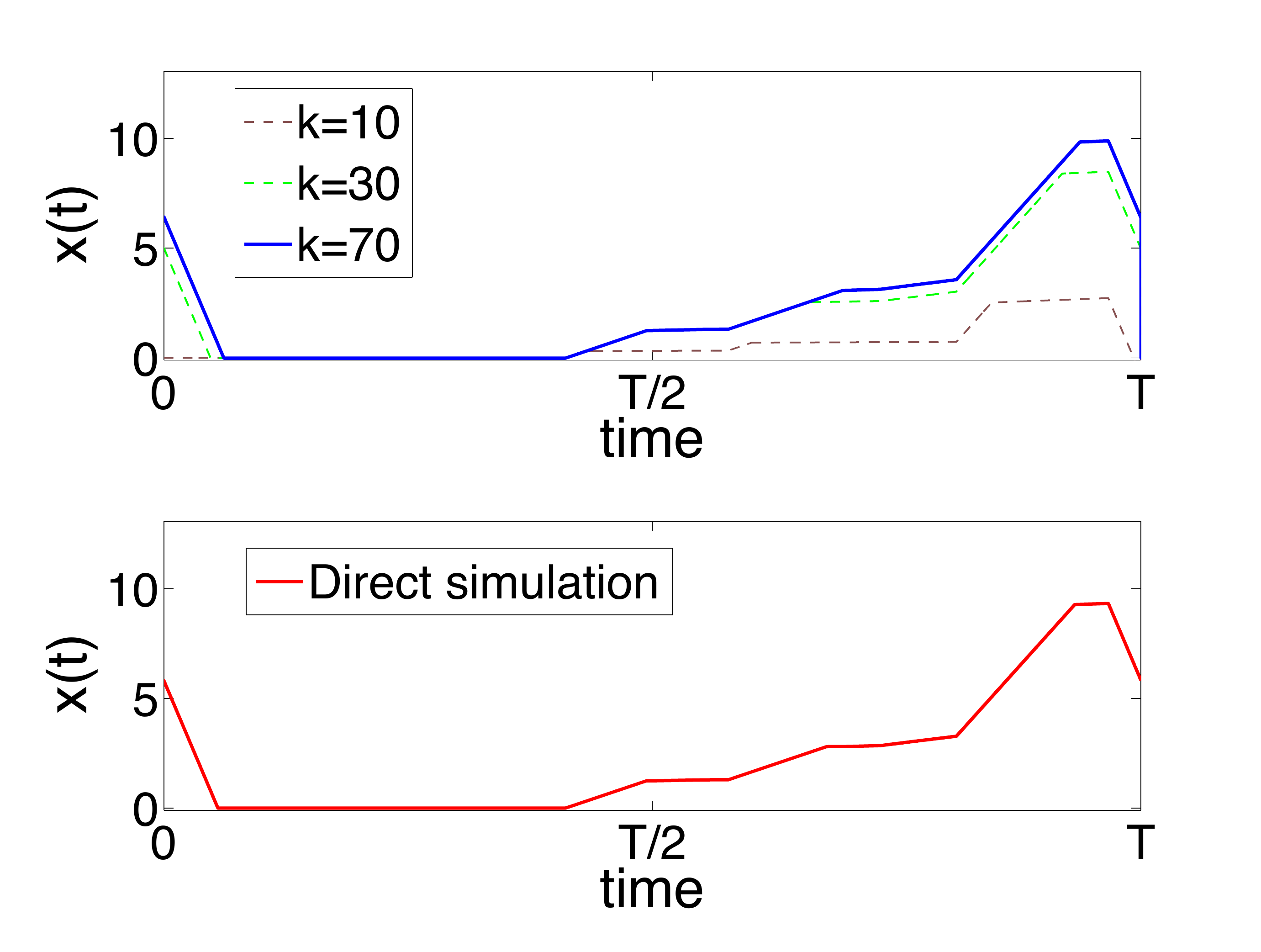} \\
(a) \hspace{1.85in}
(b) \hspace{1.85in}
(c)
\caption{\small \sf  Evolution of RMSE between (a) $x^{(k)}$ and $x^*$, and (b) $z^{(k)}$ and $z^*$ for a few representative links. (c) (top) $x_i^{(k)}$ from a few representative iterations and (bottom) $x_i^*$, both for $i=22$.
}
\vspace{-0.3in}
\label{fig-LA-rmse}
\end{center}
\end{figure}

We further compare the queue length obtained from Algorithm~\ref{alg:steady-state-computation-network} with microscopic traffic simulations in PTV VISSIM run for a 2-hour scenario starting from zero initial condition.
Figure~\ref{fig-LA-queue} compares the queue length from the last 20 cycles in VISSIM simulations, and steady-state computations from Algorithm~\ref{alg:steady-state-computation-network}, for a few representative links. For the queue lengths from VISSIM, the figure plots the mean queue length (obtained from last 20 cycles), as well as one standard deviation represented by the error bars.  In spite of the fact that the dynamical model in \eqref{eq:flow-model-main} is a coarse approximation of the microscopic traffic dynamics, e.g., \eqref{eq:flow-model-main} neglects spillbacks due to finite queue length capacity, neglects dependency of link travel times on queue lengths, and utilizes a simplified abstraction of capacity function in the form of a rectangular pulse, 
the plots in Figure~\ref{fig-LA-queue} show good consistency between the steady-state corresponding to the two models. 


\begin{figure}
\begin{center}
\includegraphics[width=2.01 in]{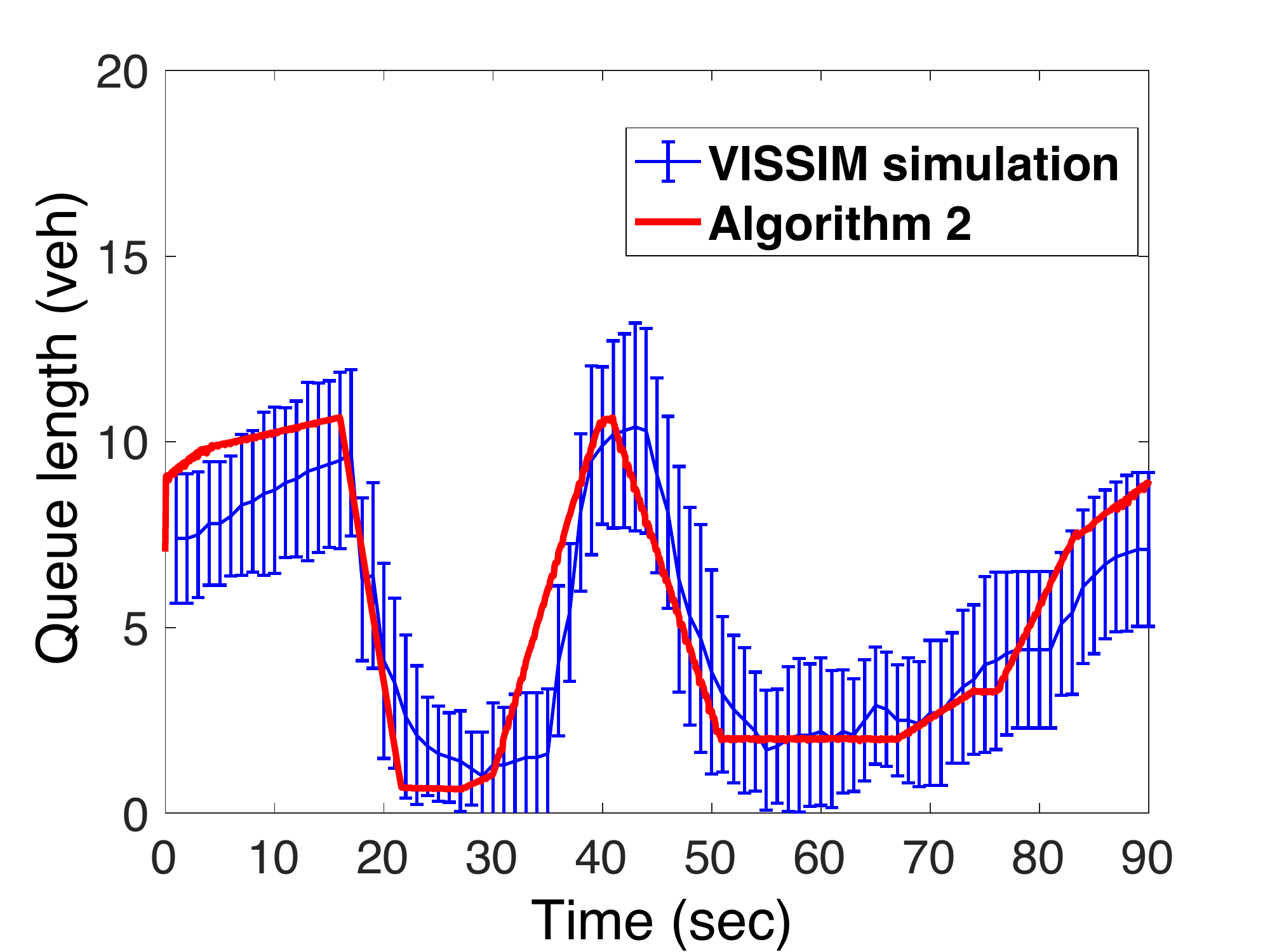} 
\hspace{0.051in}
\includegraphics[width=2.01 in]{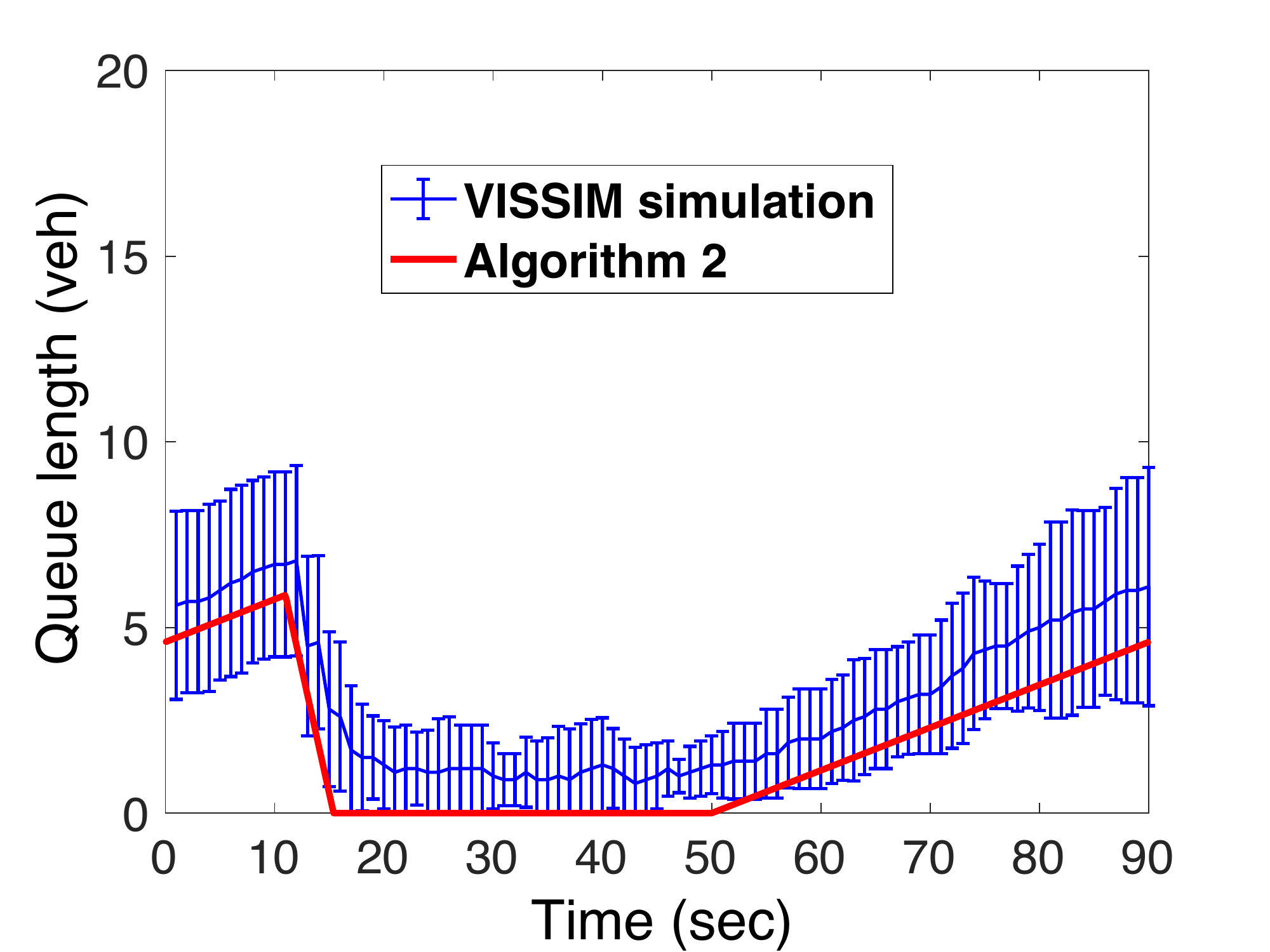} 
\hspace{0.051in}
\includegraphics[width=2.01 in]{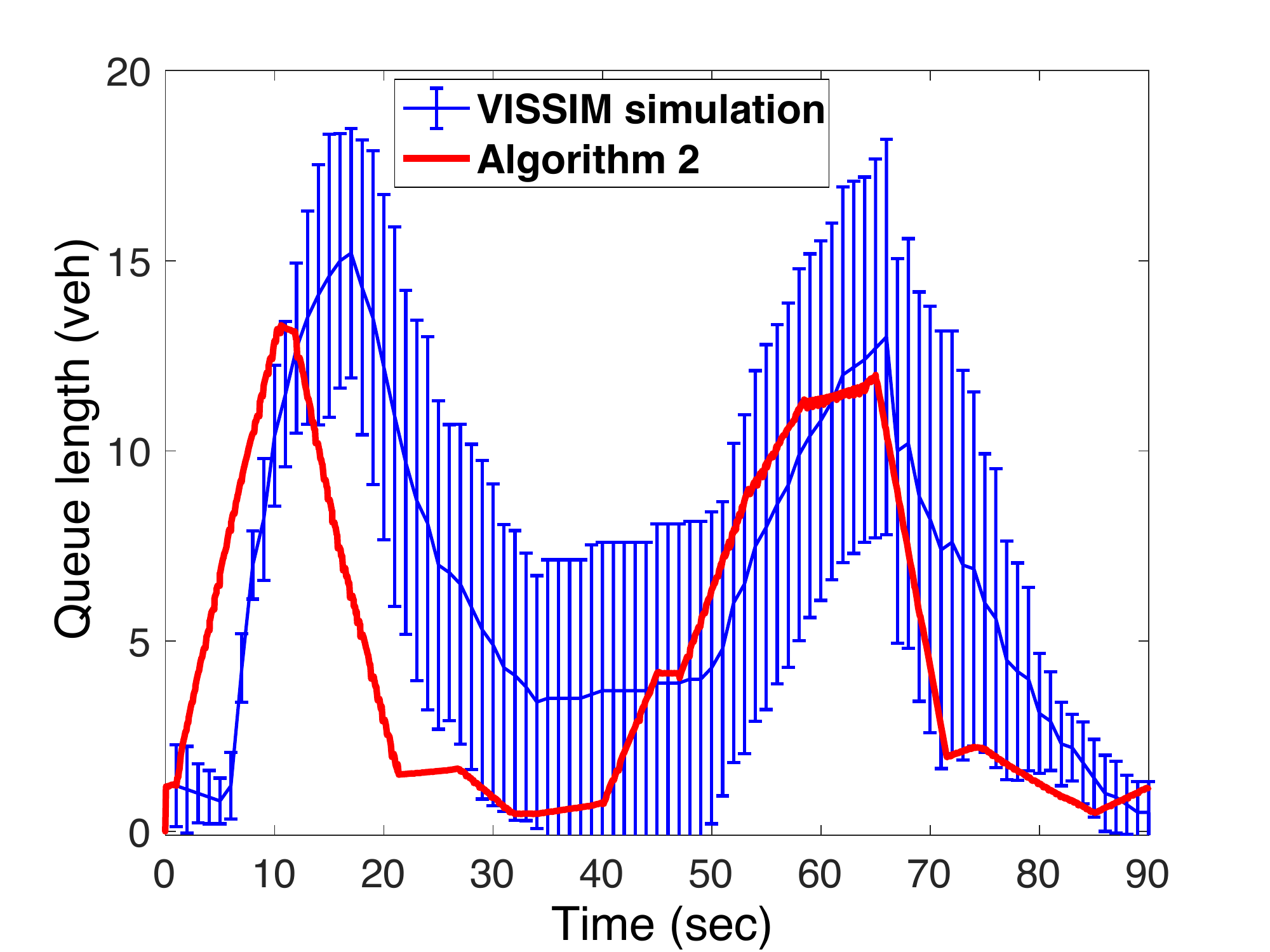} \\
(a) \hspace{1.85in}
(b) \hspace{1.85in}
(c)
\caption{\small \sf  Comparison of queue length obtained from Algorithm~\ref{alg:steady-state-computation-network} and VISSIM simulations for a few representative links (a) $i=20$, (b) $i=36$ and (c) $i=15$, in the network shown in Figure~\ref{fig-LA-network}. Here,``queue length" for link $i$ is equal to $x^*_i(t) + \sum_{j \in \mc E} R_{ji} \int_{\delta_{ji}}^0 z^*_j(t-s) \, ds$, i.e., it corresponds to the number of vehicles on link $i$ which are stationary as well as in transit from upstream.
}
\vspace{-0.3in}
\label{fig-LA-queue}
\end{center}
\end{figure}

\section{Conclusions and Future Work}
\label{sec:conclusions}
In this paper, we proposed a delay differential equation framework to simulate queue length dynamics for signalized arterial networks.
 Under periodicity and stability conditions, existence of a globally attractive periodic orbit is established for fixed-time control. An iterative procedure is also provided to compute this periodic orbit without direct simulations. Collectively, these results provide useful computational tools to evaluate the performance of signalized arterial networks for given traffic signal control parameters.  

%
 
The fact that the well-posedness of our traffic flow model does not require link travel times to be strictly bounded away from zero motivates us to consider extensions to state-dependent link travel times in order to model spillbacks. While the well-posedness of the model proposed in this paper extends to a reasonable class of adaptive control policies, extensions of the steady-state analysis and computation remains to be done. 
We plan to leverage analysis from our previous work~ \cite{Nilsson.Hosseini.ea.CDC15} for this purpose. We also plan to design traffic signal control optimization techniques which use representation of steady state queue lengths in terms of transition points as developed in this paper. Finally, feedback control, possibly in a distributed manner, along the lines of recent work for green time split control, e.g., see \cite{Varaiya:13,Hosseini.Savla:TRB16}, would greatly facilitate scalability.

\bibliographystyle{ieeetr}
\bibliography{ref1.bib,ref2.bib}

\appendix

\subsection{Proof of Proposition~\ref{prop:solution-properties}}
\label{sec:existence}

The feasible set for \eqref{z-linprog} is non-empty ($z=0$ is always feasible) and compact. Therefore, there exists at least one solution, say $\hat{z}$, to \eqref{z-linprog}. 

We first note that if $z$ and $\tilde{z}$ satisfy the constraints in \eqref{z-linprog}, then so does $z^{\text{max}}$ defined by $z_i^{\text{max}}=\max\{z_i,\tilde{z}_i\}$ for all $i \in \mc E$. This is because $z \leq c(t)$ and $\tilde{z} \leq c(t)$ implies that $z^{\text{max}} \leq c(t)$, and therefore the first inequality in \eqref{z-linprog} is trivially satisfied by $z^{\text{max}}$. With respect to the second inequality, fix some $i \in \mc I$, and let $z_i^{\text{max}}=z_i$ (without loss of generality). 
Then, 
$
z_i^{\text{max}}=z_i \leq \tilde{\lambda}_i + \sum_{j \in \mc E_i: \, z_j^{\text{max}}=z_j} R_{ji} z_j + \sum_{j \in \mc E_i: \, z_j^{\text{max}}=\tilde{z}_j} R_{ji} z_j \leq \tilde{\lambda}_i +\sum_{j \in \mc E_i: \, z_j^{\text{max}}=z_j} R_{ji} z^{\text{max}}_j + \sum_{j \in \mc E_i: \, z_j^{\text{max}}=\tilde{z}_j} R_{ji} z^{\text{max}}_j = \tilde{\lambda}_i + \sum_{j \in \mc E_i} R_{ji}  z_j^{\text{max}}$, 
where the first inequality follows from \eqref{z-linprog} and the second one follows from the definition of $z^{\text{max}}$. This argument is used to prove that $\hat{z}$ is unique and is independent of $\eta \in \real_{>0}^{\mc E}$ as follows: 
\begin{itemize}
\item[(a)] \underline{Uniqueness for a given $\eta$}: Let $z$ and $\tilde{z}$ be two optimal solutions for a given $\eta$. Since $z \neq \tilde{z}$, there exist $i, j \in \mc E$ such that $ 
z_i^{\text{max}} > z_i$ and $ 
z_j^{\text{max}} > \tilde{z}_j$. Therefore, $\eta^T z^{\text{max}} > \eta^T z = \eta^T \tilde{z}$, contradicting optimality of $z$ and $\tilde{z}$. 

\item[(b)] \underline{Independence w.r.t. $\eta$}: Let $z$ and $\tilde{z}$ be the unique optimal solutions corresponding to $\eta$ and $\tilde{\eta}$ respectively. However, using the argument in case (a) above where $z_i^{\text{max}} > z_i$ and $z_j^{\text{max}} > \tilde{z}_j$, we have $\eta^T z^{\text{max}} > \eta^T z$ and $\tilde{\eta}^T z^{\text{max}} > \tilde{\eta}^T \tilde{z}$, thereby contradicting optimality of $z$ and $\tilde{z}$. 
\end{itemize}
In order to prove the first part of \eqref{eq:z-def}, let $\hat{z}_i < c_i(t)$ for some $i \in \mc E \setminus \mc I(x)$. Then, a small increase in $\hat{z}_i$ will trivially maintain feasibility of the first set of constraints in \eqref{z-linprog}, and also maintains feasibility with respect to the second set of constraints because it only affects the right hand side which increases with increase in $\hat{z}_i$. However, increasing $\hat{z}_i$ strictly increases the objective, thereby contradicting optimality of $\hat{z}$. 

With regards to the second part of \eqref{eq:z-def}, if $c_i(t) < \tilde{\lambda}_i(t)+\sum_{j \in \mc E_i} R_{ji} \hat{z}_j$ for some $i \in \mc I(x)$, then the proof follows along the same lines as the first part of \eqref{eq:z-def}. Allowing $\hat{z}_i < \tilde{\lambda}_i(t) + \sum_{j \in \mc E_i} R_{ji} \hat{z}_j < c_i(t)$ for some $i \in \mc I(x)$ leads to a contradiction for similar reason. 

\subsection{Proof of Proposition~\ref{prop:existence-traffic-dynamics-solution}}
Once the solution $(x(t),\beta(t))$ to \eqref{eq:flow-model-main} is proven to exist and be unique, its non-negativity follows from the constraint on $z(x,t)$ in \eqref{z-linprog}.
Our approach to showing the existence and uniqueness of solution to \eqref{eq:flow-model-main} is to show it on contiguous intervals $[0,\triangle), [\triangle, 2 \triangle), \ldots$. $\triangle >0$ is chosen to be (the greatest) common divisor of: (i) time instants in $[-\bar{\delta},0]$ corresponding to switch in values of $z(t)$; (ii) time instants in $[0,T]$ corresponding to switch in values of $\lambda(t)$ and $c(t)$; and (iii) $\{\delta_{ji}\}_{j,i \in \mc E}$. Under Assumption~\ref{ass:piece-wise-constant}, such a $\triangle >0$ exists if, e.g., the three types of quantities are all rational numbers. 


The next result establishes the required existence and uniqueness of solution to \eqref{eq:flow-model-main} over $[0,\triangle)$, along with an important \emph{input-output} property. 

\begin{lemma}
\label{lem:existence-uniqueness-small-interval}
If $\map{\tilde{\lambda}}{[0,\triangle)}{\real_{\geq 0}^{\mc E}}$ is piece-wise constant and non-increasing, and $\map{c}{[0,\triangle)}{\real_{\geq 0}^{\mc E}}$ is constant,
then, for any $x(0) \in \real_{\geq 0}^{\mc E}$, there exists a unique solution $\map{x}{[0,\triangle)}{\real_{\geq 0}^{\mc E}}$ to \eqref{eq:flow-model-main}. Moreover, $\map{z}{[0,\triangle)}{\real_{\geq 0}^{\mc E}}$ is piece-wise constant and non-increasing.
\end{lemma}
\begin{proof}
\eqref{z-linprog} and Proposition~\ref{prop:solution-properties} imply that $z(x,t)$ remains constant over a time interval if so do $\tilde{\lambda}(t)$, $c(t)$ and $\mc I(x)$. 
Let $(\tau_1, \tau_2, \ldots,) \in (0,\triangle)$ be the finite number of time instants corresponding to changes in the value of $\tilde{\lambda}(t)$. 
Since $\tilde{\lambda}(t)$ and $c(t)$ are constant over $[0,\tau_1)$, $z(x,t)$ will remain constant at least until say at $t_s \in [0,\tau_1)$ when there is possibly a change in $\mc I(x)$. This also implies the existence and uniqueness of solution to \eqref{eq:flow-model-main} over $[0,t_s)$. Moreover, since $z(x,t)$ is constant over $[0,t_s)$, if $x_i(t=0^+)=0$ for some $i \in \mc E$, then $x_i(t) \equiv 0$ over $[0,t_s]$. Thus, a change in the set $\mc I(x)$ at $t_s$ could only involve its expansion. Therefore, \eqref{z-linprog} implies that $z(x(t_s),t_s) \leq z(x(t_s^-),t_s^-)$. Continuing along these lines, $\mc I(x)$ is non-contracting over $[0,\tau_1)$. Since $\tilde{\lambda}(\tau_1) < \tilde{\lambda}(\tau_1^-)$ by assumption, \eqref{z-linprog} implies that  
$\mc I(x(\tau_1^-)) \subsetneq \mc I(x(\tau_1))$, and hence also $z(x(\tau_1),\tau_1) \leq z(x(\tau_1^-),\tau_1^-)$. Collecting these facts together implies that $\mc I(x)$ is non-contracting and $z(x,t)$ is non-increasing over $[0,\triangle)$. Combining this with the fact that $\mc I(x)$ can take at most $2^{\mc E}$ distinct values, implies that the total number of changes in $\mc I(x)$ over $[0,\triangle)$ are finite. Concatenating the unique solutions to \eqref{lem:existence-uniqueness-small-interval} from between changes in $\mc I(x)$ gives the lemma.
\end{proof}

Since $z(t)$ is non-increasing and piece-wise constant in each of the intervals $[-\bar{\delta},-\bar{\delta}+\triangle), \ldots, [0,\triangle)$ (cf. Lemma~\ref{lem:existence-uniqueness-small-interval} and the assumption in Proposition~\ref{prop:existence-traffic-dynamics-solution}), this implies that $\tilde{\lambda}(t)$ is non-increasing and piece-wise constant over $[\triangle,2 \triangle)$. One can then use Lemma~\ref{lem:existence-uniqueness-small-interval} to show existence and uniqueness of solution over $[\triangle,2 \triangle)$. Recursive application of the procedure then proves Proposition~\ref{prop:existence-traffic-dynamics-solution}.

\subsection{Proof of Theorem~\ref{thm:globally-attractive-existence}}
\label{sec:attractivity}
The structure of the proof of Theorem~\ref{thm:globally-attractive-existence} follows closely along the lines of \cite{Muralidharan.Pedarsani.ea:15}, with differences due to the combination of dynamics in \eqref{eq:flow-dynamics}, and the fact that we allow $\underline{\delta}=0$.


\begin{lemma}
\label{lem:monotonicity}
Suppose $x_{0} \leq x'_{0}$, $\lambda(t) \leq \lambda'(t)$ and $c(t)=c'(t)$ for all $t \geq 0$, and $\{z(t), t \in [-\bar{\delta},0)\} \leq \{z'(t), t \in [-\bar{\delta},0)\}$. If $\{\lambda(t), \lambda'(t), c(t), c'(t)\}$ are all piecewise constant,  
then the corresponding solutions to \eqref{eq:flow-model-main} satisfy $x(t) \leq x'(t)$ and $z(t) \leq z'(t)$ for all $t \geq 0$. 
%
\end{lemma}
\begin{proof}
It suffices to show the result for a small time interval starting from zero. Moreover, it is sufficient to show that $x(t) \leq x'(t)$ in this interval, since this implies $\mc I(x'(t)) \subseteq \mc I(x(t))$, and hence $z(t) \leq z'(t)$ along the same lines as the proof of Lemma~\ref{lem:existence-uniqueness-small-interval}.
%
%
We shall prove this for each component of $x(t)$ and $x'(t)$ independently. 
Fix a component $i \in \mc E$. 
\begin{enumerate}
\item If $0<x_{0,i} \leq x'_{0,i}$, then, recalling \eqref{eq:E-i-def-new}, $\dot{x}_i(t=0) = \lambda_i(0) + \sum_{j \in \mc E \setminus \mc E_i} R_{ji} z_j(-\delta_{ji}) + \sum_{j \in \mc E_i} R_{ji} z_i(t=0)-c_i(0) \leq \lambda'_i(0) + \sum_{j \in \mc E \setminus \mc E_i} R_{ji} z'_j(-\delta_{ji}) + \sum_{j \in \mc E_i} R_{ji} z'_j(t=0) - c'_i(t)=\dot{x}'_i$, where we have used the fact that $x_0 \leq x'_0$ implies $z(0) \leq z'(0)$. Hence $x(t) \leq x'(t)$ for small time interval starting from zero.
\item Now consider the case when $0 = x_{0,i} \leq x'_{0,i}$. If $x_i(t) \equiv 0$ for a small interval starting from zero, then trivially $x_i(t) \leq x'_i(t)$ over that interval. Otherwise, the proof follows along the same lines as Case 1. 
%
\end{enumerate}
\end{proof}

\begin{lemma}
\label{lem:boundedness}
If $\lambda(t)$, $c(t)$, and $(x(0),\beta(0))$ satisfy Assumption~\ref{ass:piece-wise-constant}, and if the stability condition in Definition~\ref{def:stability} holds true, then the solution $x(t)$ to \eqref{eq:flow-model-main} is bounded.
\end{lemma}
\begin{proof}
\eqref{eq:flow-dynamics} can be rewritten as $\dot{x}_i=\lambda_i(t) + \sum_{j \in \mc E} R_{ji} z_j(t) - z_i(t) + \triangle_i(t)$, where $\triangle_i(t) = \sum_{j \in \mc E} R_{ji} \left(z_j(t-\delta_{ji}) - z_j(t) \right)$. Therefore, $\int_s^t \triangle_i(r) \, dr = \sum_{j \in \mc E} R_{ji} \left(\int_{s-\delta_{ji}}^s z_j(r) - \int_{t -\delta_{ji}}^t z_j(r) \right) \, dr$. Since $\delta_{ji} \leq \bar{\delta}$ and $z_j(r) \leq c_j(r)$ is bounded, it follows that $|\int_s^t \triangle(r) \, dr| \leq d \onebf$ for some constant $d >0$.
Suppose $x_i(t_0)>NT \bar{c}_i$ for some constant $N > \frac{d}{T \epsilon}$ and $t_0 \geq 0$, where $\epsilon>0$ is from Definition~\ref{def:stability}. Since $x_i(t+T)-x_i(t) \geq -T \bar{c}_i$, we have $x_i(t)>0$ for $t_0 \leq t \leq t_0 + NT$. Therefore, 
$$
x_i(t_0+NT) - x_i(t_0) \leq NT \bar{\lambda}_i + NT \sum_{j \in \mc E} R_{ji} \bar{c}_j - NT \bar{c}_i + d  \leq - NT \epsilon + d < 0
$$
where we use the stability condition from Definition~\ref{def:stability} in the second inequality. 
This is sufficient to show that $x_i(t)$ is bounded, since the queue length increments per cycle are upper bounded. 
\end{proof}
We next state an important result on \emph{contraction} and a global attractivity property of \eqref{eq:flow-dynamics}. 
\begin{proposition}
\label{prop:contraction}
Let the conditions in Proposition~\ref{prop:existence-traffic-dynamics-solution} hold true. If $(x(t),\beta(t))$ and $(\tilde{x}(t),\tilde{\beta}(t))$ denote the trajectories starting from $(x_0,\beta_0)$ and $(\tilde{x}_0,\tilde{\beta}_0)$ respectively, then
\begin{equation}
\label{eq:contraction}
\|x(t)-\tilde{x}(t)\|_1 + \|\beta(t)-\tilde{\beta}(t)\|_1  \leq \|x_0-\tilde{x}_0\|_1 + \|\beta_0 - \tilde{\beta}_0\|_1
\end{equation}
Moreover, if the stability condition in Definition~\ref{def:stability} is satisfied, then
\begin{equation}
\label{eq:convergence}
\lim_{t \to \infty} \|x(t)-\tilde{x}(t)\|_1 = 0, \qquad \lim_{t \to \infty} \|\beta(t) - \tilde{\beta}(t)\|_1 = 0
\end{equation}
\end{proposition}
\begin{proof}
Let $\bar{x}_{0,i}:=\max\{x_{0,i},\tilde{x}_{0,i}\}$ and $\underline{x}_{0,i} :=\min\{x_{0,i},\tilde{x}_{0,i}\}$, for all $i \in \mc E$. Let $\bar{\beta}_0$ and $\underline{\beta}_0$ be defined similarly. 
Therefore, $\|x_0-\tilde{x}_0\|_1 = \sum_{i \in \mc E} |x_{0,i}-\tilde{x}_{0,i}| = \sum_{i \in \mc E} (\bar{x}_{0,i}-\underline{x}_{0,i})=\|\bar{x}_0-\underline{x}_0\|_1$. Similarly, $\|\beta_0-\tilde{\beta}_0\|_1 = \|\bar{\beta}_0-\underline{\beta}_0\|_1$. Let $(\bar{x}(t),\bar{\beta}(t))$ and $(\underline{x}(t),\underline{\beta}(t))$ be the trajectories starting from $(\bar{x}_0,\bar{\beta}_0)$ and $(\underline{x}_0,\underline{\beta}_0)$ respectively. 
Lemma~\ref{lem:monotonicity} then implies that 
$\underline{x}(t) \leq x(t) \leq \bar{x}(t)$, $\underline{x}(t) \leq \tilde{x}(t) \leq \bar{x}(t)$, $\underline{\beta}(t) \leq \beta(t) \leq \bar{\beta}(t)$, and $\underline{\beta}(t) \leq \tilde{\beta}(t) \leq \bar{\beta}(t)$, which then implies that $\|x(t)-\tilde{x}(t)\|_1 \leq \|\bar{x}(t)-\underline{x}(t)\|_1$ and $\|\beta(t)-\tilde{\beta}(t)\|_1 \leq \|\bar{\beta}(t)-\underline{\beta}(t)\|_1$. Therefore, it suffices to show 
\begin{equation}
\label{eq:continuity-monotonicity-equiv-new}
\|\bar{x}(t)-\underline{x}(t)\|_1 + \|\bar{\beta}(t)-\underline{\beta}(t)\|_1 \leq \|\bar{x}_0-\underline{x}_0\|_1+ \|\bar{\beta}_0-\underline{\beta}_0\|_1
\end{equation}
We show \eqref{eq:continuity-monotonicity-equiv-new} using an intuitive argument similar to the one used in \cite[Lemma 2]{Muralidharan.Pedarsani.ea:15}. 
Color the vehicles in the initial state (i.e., $(\bar{x}_0,\bar{\beta}_0)$ and $(\underline{x}_0,\underline{\beta}_0)$) red, and all the vehicles arriving after that as black. Therefore, 
the right hand side in \eqref{eq:continuity-monotonicity-equiv-new} represents the excess red vehicles initially in the system with the larger initial condition. Subsequently, in each queue, there will be black and red vehicles. Change the service discipline in each queue so that all black vehicles are served ahead of every red vehicle. This has two implications. First, since the service times for red vehicles in each queue are the same in each of the two systems, every red vehicle common to both the systems receives identical service in the two systems. That is, more red vehicles depart the system starting from $(\bar{x}_0,\bar{\beta}_0)$ than in the system starting from $(\underline{x}_0,\underline{\beta}_0)$, i.e., $\sum_{i \in \mc E}(\bar{x}_{0,i} + \sum_{j \in \mc E} R_{ji} \int_{-\delta_{ji}}^0 \bar{z}_i(s) \, ds) - \sum_{i \in \mc E} \left( \bar{x}_i^{\text{red}}(t) + \sum_{j \in \mc E} R_{ji} \int_{t-\delta_{ji}}^t \bar{z}_i(s) \, ds \right)\geq \sum_{i \in \mc E} (\underline{x}_{0,i} + \sum_{j \in \mc E} R_{ji} \int_{t-\delta_{ji}}^t \underline{z}_i(s) \, ds) - \sum_{i \in \mc E} \left(\underline{x}_i^{\text{red}}(t) + \sum_{j \in \mc E} R_{ji} \int_{t-\delta_{ji}}^t \underline{z}^{\text{red}}_i(s) \, ds\right)$, i.e., 
\begin{align}
\sum_{i \in \mc E} \left( \bar{x}_i^{\text{red}}(t) -  \underline{x}_i^{\text{red}}(t)\right) + \sum_{i \in \mc E} \sum_{j \in \mc E} R_{ji} \int_{t-\delta_{ji}}^t & \left(\bar{z}_i^{\text{red}}(s) - \underline{z}_i^{\text{red}}(s) \right) \, ds \nonumber \\ 
& \leq \sum_{i \in \mc E} \left( \bar{x}_{0,i} - \underline{x}_{0,i} \right)+ \sum_{i \in \mc E} \sum_{j \in \mc E} R_{ji} \int_{-\delta_{ji}}^0 \left(\bar{z}_i(s) -  \underline{z}_i(s) \right) \, ds \quad \text{i.e., } \nonumber \\
\|\bar{x}^{\text{red}}(t)-\underline{x}^{\text{red}}(t)\|_1 + \|\bar{\beta}^{\text{red}}(t)-\underline{\beta}^{\text{red}}(t)\|_1 & \leq \|\bar{x}_0-\underline{x}_0\|_1+ \|\bar{\beta}_0-\underline{\beta}_0\|_1
\label{eq:red-monotonicity}
\end{align}
Second, service of black vehicles is unaffected by red vehicles in both the systems. Therefore, the  number of black vehicles in each queue, and in particular, the total number of black vehicles in the entire network for both the systems are the same at any time. This combined with \eqref{eq:red-monotonicity} gives \eqref{eq:continuity-monotonicity-equiv-new}.

In order to prove \eqref{eq:convergence}, let $(\hat{x}(t),\hat{\beta}(t))$ be the trajectory starting from the initial condition $(0,0)$. Since $\|x(t)-\tilde{x}(t)\|_1 \leq \|x(t)-\hat{x}(t)\|_1 + \|\tilde{x}(t)-\hat{x}(t)\|_1$ and $\|\beta(t)-\tilde{\beta}(t)\|_1 \leq \|\beta(t)-\hat{\beta}(t)\|_1 + \|\tilde{\beta}(t)-\hat{\beta}(t)\|_1$, it suffices to prove \eqref{eq:convergence} for $(\tilde{x}(0),\tilde{\beta}(0))=(\hat{x}(0),\hat{\beta}(0))=(0,0)$. Using the red vehicle terminology from before, $\|x(t)-\tilde{x}(t)\|_1 + \|\beta(t)-\tilde{\beta}(t)\|_1$ then denotes the number of red vehicles in $(x(t),\beta(t))$. Stability condition implies that all red vehicles eventually leave the network, i.e., $\lim_{t \to \infty} \left( \|x(t) - \tilde{x}(t)\|_1 + \|\beta(t)-\tilde{\beta}(t)\|_1\right) = 0$.
\end{proof}

We can now finish the proof of Theorem~\ref{thm:globally-attractive-existence} as follows. Consider the 
trajectory starting from $\left(x(0), \beta(0)\right)=(0,0)$. In particular, consider the sequence of following points on this trajectory: $\{\left(x(nT), \beta(nT) \right) \}_{n=0}^{\infty}$. Monotonicity (Lemma~\ref{lem:monotonicity}) and boundedness (Lemma~\ref{lem:boundedness}) implies that this sequence converges, say to $(x^*,\beta^*)$. We now establish that the trajectory starting from such a point is periodic. This, together with global attractivity implied by \eqref{eq:convergence}, then establishes Theorem~\ref{thm:globally-attractive-existence}, i.e., every trajectory converges to the periodic trajectory starting from $(x^*,\beta^*)$. 

 If $F(x((n-1)T),\beta((n-1)T))=(x(nT),\beta(nT))$ denotes the associated Poincare map, then the desired periodicity is equivalent to showing $F(x^*,\beta^*)=(x^*,\beta^*)$, i.e., $(x^*(T),\beta^*(T))=(x^*,\beta^*)$, i.e., $\lim_{n \to \infty} (x((n+1)T),\beta((n+1)T))=(x^*,\beta^*)$, i.e., $\lim_{n \to \infty} F(x(nT),\beta(nT))=(x^*,\beta^*)$. 
A sufficient condition for this is continuity of $F$, which follows from \eqref{eq:contraction}.

\subsection{Technical Corollary}
\label{sec:monotonicity}
The following corollary to Lemma~\ref{lem:monotonicity} and Proposition~\ref{prop:contraction} is used in Section~\ref{sec:steady-state-network}.
\begin{corollary}
\label{cor:input-output-monotonicity}
Consider an isolated link $i$ with $T$-periodic capacity function $c_i(t)$.
Let $y_i(t)$ and $y_i'(t)$ be $T$-periodic inflow functions, both satisfying the stability condition in Definition~\ref{def:stability}, and $y_i(t) \leq y'_i(t)$ for all $t \in [0,T]$. If the corresponding steady state $T$-periodic queue lengths are $x^*_i(t)$ and ${x^*}_i'(t)$ respectively, and the steady state $T$-periodic link outflows are $z^*_i(t)$ and ${z^*}'_i(t)$ respectively, then $x^*_i(t) \leq {x^*}'_i(t)$ and $z^*_i(t) \leq {z^*}'_i(t)$ for all $t \in [0,T]$.
\end{corollary}
\begin{proof}
Let $(x(t),z(t))$ and $(x'(t),z'(t))$ be the system trajectories for the two systems, both starting from initial condition $(x_0,\beta_0)$. Lemma~\ref{lem:monotonicity} implies that $x'(t) \geq x(t)$ and $z'(t) \geq z(t)$ for all $t \geq 0$. On the other hand, Theorem~\ref{thm:globally-attractive-existence} implies that $(x(t),z(t))$ and $(x'(t),z'(t))$ converge to $T$-periodic trajectories $(x^*(t),z^*(t))$ and $({x^*}'(t),{z^*}'(t))$ respectively. Combining these facts gives the desired result. 
\end{proof}

\end{document}